\documentclass[10pt]{amsart}

\usepackage[english]{babel}

\usepackage{amssymb}
\usepackage[leqno]{amsmath}
\usepackage{mathrsfs}
\usepackage{stmaryrd}
\usepackage{chemarrow}
\usepackage[norelsize]{algorithm2e}
\usepackage{enumerate}
\usepackage{graphicx}
\usepackage[all]{xy}
\usepackage{tikz, tikz-3dplot}
\usetikzlibrary{arrows}
\usepackage{multirow}
\usepackage[colorinlistoftodos]{todonotes}
\usepackage[colorlinks=true, allcolors=blue]{hyperref}
\usepackage[hyperpageref]{backref}

\newtheorem{theorem}{Theorem}[section]
\newtheorem{lemma}[theorem]{Lemma}

\newtheorem{example}[theorem]{Example}

\newtheorem{remark}{Remark}

\renewcommand{\div}{\operatorname{div}}

\newcommand{\sym}{\operatorname{sym}}

\begin{document}
\title[Conforming Virtual Elements in Arbitrary Dimension]{$H^m$-Conforming Virtual Elements in Arbitrary Dimension}
\author{Chunyu Chen}%
\address{Hunan Key Laboratory for Computation and Simulation in Science and Engineering; School of Mathematics and Computational Science, Xiangtan University, Xiangtan 411105, P.R.China}%
\email{202131510114@smail.xtu.edu.cn}%
\author{Xuehai Huang$^{\ast}$}%
\address{School of Mathematics, Shanghai University of Finance and Economics, Shanghai 200433, China}%
\email{huang.xuehai@sufe.edu.cn}%
\author{Huayi Wei}%
\address{Hunan Key Laboratory for Computation and Simulation in Science and Engineering; School of Mathematics and Computational Science, Xiangtan University, Xiangtan 411105, P.R.China}%
\email{weihuayi@xtu.edu.cn}%


\thanks{* Corresponding author.}
\thanks{The second author was supported by the National Natural Science Foundation of
China (NSFC) (Grant No. 12171300), and the Natural Science Foundation of Shanghai
(Grant No. 21ZR1480500). The first and third authors were supported by the National Natural
Science Foundation of China (NSFC) (Grant No. 12261131501, 11871413), and the construction of innovative provinces in Hunan Province (Grant No. 2021GK1010)}

\makeatletter
\@namedef{subjclassname@2020}{\textup{2020} Mathematics Subject Classification}
\makeatother




\begin{abstract}
The $H^m$-conforming virtual elements of any degree $k$ on any shape of polytope in $\mathbb R^n$ with $m, n\geq1$ and $k\geq m$ are recursively constructed by gluing conforming virtual elements on faces in a universal way. For the lowest degree case $k=m$, the set of degrees of freedom only involves function values and derivatives up to order $m-1$ at the vertices of the polytope. The inverse inequality and several norm equivalences for the $H^m$-conforming virtual elements are rigorously proved. The $H^m$-conforming virtual elements are then applied to discretize a polyharmonic equation with a lower order term. With the help of the interpolation error estimate and norm equivalences, the optimal error estimates are derived for the $H^m$-conforming virtual element method.
\end{abstract}

\maketitle


\section{Introduction}

Recently Hu, Lin and Wu constructed $H^m$-conforming finite elements of degree $k$ on simplices in $\mathbb R^n$ with $k\geq 2^n(m-1)+1$ and $m, n\geq1$ in a unified way~\cite{HuLinWu2021}, which generalizes the finite elements in two dimensions in \cite{BrambleZlamal1970,Zenisek1970,ArgyrisFriedScharpf1968} and the finite elements in three dimensions in \cite{Zenisek1974a,Zhang2009a,Zhang2016a}. The simplical lattice is used in \cite{ChenHuang2021Cmgeodecomp} to show the geometric decomposition of smooth finite elements. 
The work \cite{HuLinWu2021} is theoretically important, and is a significant progress in the field of construction of $H^m$-conforming elements in $\mathbb R^n$.
Since polynomial shape functions are infinitely differentiable,
the $2^{n-1}(m-1)$th order derivatives of shape functions at vertices are included in the degrees of freedom (DoFs), which results in the very high polynomial degree $k\geq 2^n(m-1)+1$ for $H^m$-conforming finite elements. In \cite{Xu2020}, Xu devised $H^m$-conforming piecewise polynomials based on the artificial neural network with $k\geq m$ and then
developed a finite neuron method, whose practical value is also limited since solving the underlying non-linear and non-convex optimization problem is challenging. We refer to \cite{HuZhang2015a} for $H^m$-conforming finite elements on macro-hypercubes and \cite{FuGuzmanNeilan2020} for $H^2$-conforming finite elements on macro-simplices in arbitrary dimension. 

Alternatively, in \cite{ChenHuang2020,Huang2020} we devised $H^m$-nonconforming virtual elements of any degree $k$ on any shape of polytope $K$ in $\mathbb R^n$ with $k\geq m$ in a universal way by employing a generalized Green’s identity.
When $K$ is a simplex, $1\leq m\leq n$ and $k = m$, the virtual elements in \cite{ChenHuang2020} are exactly the nonconforming finite elements in \cite{WangXu2013,WangXu2006}.
And when $K$ is a simplex, $m=n+1$ and $k = m$, 
the DoFs of the virtual elements in \cite{Huang2020} are same as those of the nonconforming finite elements in \cite{WuXu2019}.
We refer to \cite{WuXu2017,HuZhang2017,HuZhang2019} for more $H^m$-nonconforming finite elements and~\cite{ZhaoChenZhang2016,ZhaoZhangChenMao2018,AntoniettiManziniVerani2018} for more $H^m$-nonconforming virtual elements.

We shall construct $H^m$-conforming virtual elements of any degree $k$ of polynomials on a very general polytope $K \subset\mathbb R^n$ in arbitrary dimension $n$ and any derivative order $m$ with $k\geq m$ and $m, n\geq1$ in this paper.
The $H^1$-conforming virtual elements were initially developed in~\cite{BeiraodaVeigaBrezziCangianiManziniEtAl2013,BeiraodaVeigaBrezziMariniRusso2014} in two and three dimensions.
The $H^m$-conforming virtual elements of degree $k$ for $k\geq m$ and $m\geq1$ in two dimensions have been designed in a series of works~\cite{BeiraoManzini2014,AntoniettiManziniVerani2020,AntoniettiManziniScacchiVerani2021,BrezziMarini2013}.
In three dimensions, the $H^2$-conforming virtual elements for $k\geq 2$ were devised in \cite{BeiraodaVeigaDassiRusso2020}.
When $K$ is a tetrahedron in three dimensions, by using the Argyris element \cite{ArgyrisFriedScharpf1968,BrennerSung2005} and Hermite element \cite{Ciarlet1978} on faces, $H^2$-conforming virtual elements for $k\geq 5$ were advanced in~\cite{ChenHuang2022}.
A different approach is adopted in \cite{BrennerSung2019} to construct $H^2$-conforming virtual elements on tetrahedrons.
We intend to extend these works to arbitrary spacial dimension $n$, any order $m$ of Sobolev spaces and any polynomial degree $k\geq m$. 

We construct $H^m$-conforming virtual elements $(K, \mathcal N_k^m(K), V_k^m(K))$ 
by gluing conforming virtual elements on faces recursively.
The virtual element space is defined as 
\begin{align*}
V_k^m(K):=\big\{ v\in H^m(K): &(-\Delta)^mv\in \mathbb P_{k}(K),\; (v-\Pi_k^Kv, q)_K=0\;\;\forall~q\in\mathbb P_{k-2m}^{\perp}(K),  \\
& (\nabla^jv)|_{\mathcal S_K^r}\in H^1(\mathcal S_K^r;\mathbb S_{n}(j)) \textrm{ for } j=0,1,\cdots,m-1, \\
&\frac{\partial^{|\alpha|}v}{\partial \nu_F^{\alpha}}\Big|_F\in V_{k-|\alpha|}^{m-|\alpha|}(F)\quad\forall~F\in\mathcal F^{r}(K),\\
&\qquad\qquad\quad\;\;\, r=1,\cdots, n-1, \alpha\in A_r, \textrm{ and } |\alpha|\leq m-1\big\}    
\end{align*}
with $V_{k-|\alpha|}^{m-|\alpha|}(e):=\mathbb P_{\max\{k-|\alpha|, 2(m-|\alpha|)-1\}}(e)$ for each one-dimensional edge $e\in \mathcal F^{n-1}(K)$, where the local $H^m$-projection operator $\Pi_k^K$ is introduced to ensure the $L^2$-orthogonal projection $Q_k^Kv$ is computable using only the DoFs in $\mathcal N_k^m(K)$ for any virtual function $v\in V_k^m(K)$ following the idea in \cite{AhmadAlsaediBrezziMariniEtAl2013}.
When $n\geq2$, $\mathbb P_k(K)\subseteq V_k^m(K)$ but $\mathbb P_{k+1}(K)\not\subseteq V_k^m(K)$.
The DoFs in $\mathcal N_k^m(K)$ are motivated by $\frac{\partial^{|\alpha|}v}{\partial \nu_F^{\alpha}}\Big|_F\in V_{k-|\alpha|}^{m-|\alpha|}(F)$ in the definition of $V_k^m(K)$. 
With the help of the concepts of data spaces and Whitney arrays \cite{Verchota1990},
the dimension of $V_k^m(K)$ is exactly counted by using the inverse trace theorem of $H^m(K)$ and the well-posedness of the $m$th harmonic equation with Dirichlet boundary conditions.

For the lowest degree case $k=m$, the set of DoFs $\mathcal N_m^m(K)$ is very simple, only involving  function values and derivatives up to order $m-1$ at the vertices of polytope $K$, i.e.
$$
h_K^j\nabla^{j}v(\delta) \quad\forall~\delta\in\mathcal F^{n}(K), \;j=0,1,\cdots,m-1.
$$
Here the scaling $h_K^j$ is used so that all the DoFs share the same order of magnitude.
These DoFs are even simpler than those of non-conforming virtual elements in \cite{ChenHuang2020,Huang2020}.
If furthermore $K\subset\mathbb R^n$ is a simplex, $\dim V_m^m(K)=(n+1)\dim\mathbb P_{m-1}(K)$, which is much smaller than the dimension $\dim \mathbb P_{2^n(m-1)+1}(K)$ of the lowest degree $H^m$-conforming finite element in \cite{HuLinWu2021}.
And there are no super-smooth DoFs included in $\mathcal N_k^m(K)$, i.e., all the orders of the derivatives involved in the DoFs are less than $m$.
This is one of the attractive features of virtual elements.

Another contribution of this paper is establishing the inverse inequality and norm equivalences for the $H^m$-conforming virtual elements $(K, \mathcal N_k^m(K), V_k^m(K))$ under the assumption that the polytope $K$ is star-shaped and all the diameters of all faces of $K$ are equivalent to the diameter of $K$. 
The inverse inequality for $V_k^m(K)$ is derived from the multiplicative trace inequality, the inverse trace theorem, the inverse inequality for polynomials and the mathematical induction.
Employing the inverse inequality, the trace inequality and the Poincar\'e-Friedrichs inequality,
we arrive at several norm equivalences on virtual element spaces $V_{k}^{m}(K)$, $\ker(Q_k^K)\cap V_{k}^{m}(K)$ and $\ker(\Pi_k^K)\cap V_{k}^{m}(K)$, where $\ker(T)\cap V_{k}^{m}(K):=\{v\in V_{k}^{m}(K): Tv=0\}$ with operator $T=Q_k^K$ or $\Pi_k^K$. Especially we acquire the classical $L^2$ norm equivalence as finite elements
\begin{align*}
\|v\|_{0,K}^2&\eqsim \|Q_{k-2m}^Kv\|_{0,K}^2 + \sum_{\delta\in\mathcal F^{n}(K)}\sum_{i=0}^{m-1}h_K^{n+2i}|\nabla^{i}v(\delta)|^2 \notag\\
&\quad
+ \sum_{r=1}^{n-1}\sum_{F\in\mathcal F^{r}(K)}\sum_{\alpha\in A_{r}, |\alpha|\leq m-1}h_K^{r+2|\alpha|}\Big\|Q_{k-2m+|\alpha|}^{F}\frac{\partial^{|\alpha|}v}{\partial \nu_{F}^{\alpha}}\Big\|_{0,F}^2 \quad \forall~v\in V_{k}^{m}(K),
\end{align*}
in which all terms in the right hand side completely coincide with all the DoFs in $\mathcal N_k^m(K)$. This extends the stability analysis of virtual elements in \cite{BrennerGuanSung2017,BeiraodaVeigaLovadinaRusso2017,ChenHuang2018,BrennerSung2018,HuangYu2021}.

The constructed conforming virtual elements are then applied to discretize a polyharmonic equation with a lower order term.
To analyze the conforming virtual element method, we construct a qausi-interpolation operator and derive the interpolation error estimate with the help of the norm equivalence on $V_{k}^{m}(K)$. 
Finally the optimal error estimates are presented for the conforming virtual element method.
This paper is motivated by the theoretical purposes. We also present numerical results for a fourth-order elliptic problem and a sixth-order elliptic problem in two dimensions.

The rest of this paper is organized as follows. Some notations and mesh conditions are shown in Section~\ref{sec:pre}. In Section~\ref{sec:vem} $H^m$-conforming virtual elements are constructed.
The inverse equality and several norm equivalences are proved in Section~\ref{sec:invnorm}.
In Section~\ref{sec:vempolyharmonic} the $H^m$-conforming virtual elements are applied to discretize a polyharmonic equation with a lower order term. And numerical results are provided in Section~\ref{sec:numericalexamps}.

\section{Preliminaries}\label{sec:pre}
\tdplotsetmaincoords{80}{150}

\subsection{Notation}
In this paper we will adopt the same notations as in \cite{ChenHuang2020,Huang2020}.
For any non-negative integer $r$ and $1\leq\ell\leq n$, notation $\mathbb T_{\ell}(r):=\underbrace{\mathbb R^{\ell}\otimes\cdots\otimes\mathbb R^{\ell}}_{r}$ stands for the set of $r$-tensor spaces over $\mathbb R^{\ell}$.
Introduce the symmetric $r$-tensor space
$$
\mathbb S_{\ell}(r):=\{\tau=(\tau_{i_1i_2\cdots i_r})\in\mathbb T_{\ell}(r): \tau_{i_{\sigma(1)}i_{\sigma(2)}\cdots i_{\sigma(r)}}=\tau_{i_1i_2\cdots i_r} \textrm{ for any } \sigma\in\mathfrak{S}_{r}\},
$$
where $\mathfrak{S}_{r}$ is the set of all permutations of $(1, 2, \cdots, r)$.
For tensor $\tau\in \mathbb T_{\ell}(r)$, the symmetric part of $\tau$ is a symmetric tensor in $\mathbb S_{\ell}(r)$ defined by
$$
(\sym\tau)_{i_1i_2\cdots i_r}:=\frac{1}{r!}\sum_{\sigma\in\mathfrak{S}_{r}}\tau_{i_{\sigma(1)}i_{\sigma(2)}\cdots i_{\sigma(r)}}\quad\textrm{ for } 1\leq i_1, i_2,\cdots, i_r\leq \ell.
$$
Given $r$-tensors $\tau, \varsigma\in \mathbb T_{\ell}(r)$, define the scalar product $\tau:\varsigma\in \mathbb R$ by 
\[
\tau:\varsigma:=\sum_{i_1=1}^{\ell}\cdots\sum_{i_r=1}^{\ell} \tau_{i_1,\cdots, i_{r}}\varsigma_{i_1,\cdots,i_{r}}.
\]
Denote by $\mathbb N$ the set of all non-negative integers.
For an $n$-dimensional multi-index $\alpha = (\alpha_1, \cdots , \alpha_n)$ with $\alpha_i\in\mathbb N$, define $|\alpha|:=\sum\limits_{i=1}^n\alpha_i$ and $\alpha!=\alpha_1!\cdots\alpha_n!$.
For $0\leq j \leq n$, let $A_{j}$ be the set consisting of all multi-indexes $\alpha$ with $\sum\limits_{i=j+1}^n\alpha_i=0$, i.e., non-zero index only exists for $1\leq i\leq j$.

Let $\Omega\subset \mathbb{R}^n~(n\geq 1)$ be a bounded
polytope with positive integer $n$.
Given a bounded domain $G\subset\mathbb{R}^{n}$ and a
non-negative integer $k$, let $H^k(G; \mathbb X)$ be the usual Sobolev space of functions
over $G$ taking values in the tensor space $\mathbb X$ for $\mathbb X=\mathbb T_{\ell}(r), \mathbb S_{\ell}(r)$, whose norm and semi-norm are denoted by
$\Vert\cdot\Vert_{k,G}$ and $|\cdot|_{k,G}$ respectively. Set $H^k(G):=H^k(G; \mathbb T_{\ell}(0))$.
Define $H_0^k(G)$ as the closure of $C_{0}^{\infty}(G)$ with
respect to the norm $\Vert\cdot\Vert_{k,G}$.
Let $(\cdot, \cdot)_G$ be the standard inner product on $L^2(G; \mathbb X)$. If $G$ is $\Omega$, we abbreviate
$\Vert\cdot\Vert_{k,G}$, $|\cdot|_{k,G}$ and $(\cdot, \cdot)_G$ by $\Vert\cdot\Vert_{k}$, $|\cdot|_{k}$ and $(\cdot, \cdot)$,
respectively.
Denote by $h_G$ the diameter of $G$.
Let $\mathbb P_k(G)$ be the set of all
polynomials over $G$ with the total degree no more than $k$, whose tensorial version space is denoted by $\mathbb P_{k}(G; \mathbb X)$. Let $\mathbb P_k(G):=\{0\}$ if $k<0$.
Let $Q_k^{G}$ be the $L^2$-orthogonal projection onto $\mathbb P_k(G; \mathbb X)$. For a function $v$, $Q_k^{G}v$ is understood as $v|_G$ when $G$ is a point, whether $k$ is non-negative or negative.
For non-negative integers $k$ and $m$, let $\mathbb P_{k-2m}^{\perp}(G)\subseteq\mathbb P_{k}(G)$ be the orthogonal complement space of $\mathbb P_{k-2m}(G)$ of $\mathbb P_{k}(G)$ with respect to the inner product $(\cdot,\cdot)_G$.
Denote by $\#S$ the number of elements in a finite set $S$.

Let $\{\mathcal {T}_h\}$ be a family of partitions
of $\Omega$ into nonoverlapping simple polytopal elements with $h:=\max_{K\in \mathcal {T}_h}h_K$.
Let $\mathcal{F}_h^r$ be the set of all $(n-r)$-dimensional faces
of the partition $\mathcal {T}_h$ for $r=1, 2, \cdots, n$.
For simplicity, let $\mathcal{F}_h^0:=\mathcal {T}_h$.
Moreover, we set for each $K\in\mathcal{T}_h$
\[
\mathcal{F}^r(K):=\{F\in\mathcal{F}_h^r: F\subset\partial K\}.
\]
The supscript $r$ in $\mathcal{F}_h^r$ represents the co-dimension of an $(n-r)$-dimensional face $F$. 
Similarly, we define
\[
\mathcal F^s(F):=\{e\in\mathcal{F}_h^{r+s}: e\subset\overline{F}\}.
\]
Here $s$ is the co-dimension relative to the face $F$. 
For any $F\in\mathcal F_h^r$ with $r=0,1,\cdots, n-2$,
let the $(n-r-s)$-dimensional skeleton $\mathcal S_F^s$ be the union of all faces in $\mathcal F^{s}(F)$ for $s=1,\cdots,n-r-1$.


For any $F\in\mathcal{F}_h^r$ with $1\leq r\leq n-1$, let $\nu_{F,1}, \cdots, \nu_{F,r}$ be its
mutually perpendicular unit normal vectors, and $t_{F,1}, \cdots, t_{F,n-r}$ be its
mutually perpendicular unit tangential vectors.
We abbreviate $\nu_{F,1}$ as $\nu_{F}$ when $r=1$, and $t_{F,1}$ as $t_{F}$ when $r=n-1$. We refer to Fig.~\ref{fig:normaltangentialvectors} for an example of normal vectors and tangential vectors.
\begin{figure}[htbp]
  \centering
\begin{tikzpicture}[%
    tdplot_main_coords,
    scale=1.5,
    >=stealth]   
    \draw[dashed] (-2,1,0)--(0,0,0) -- (2,0,0);
    \draw (2,0,0) -- (2,2,0) -- (0,2,0) -- (-2,1,0);
    \draw (2,0,0) -- (2,2,0) -- (2,2,2) -- (2,0,2) -- cycle;
    \draw (0,2,0) --(0,2,2) --  (2,2,2);
    \draw[dashed] (2,0,0) -- (0,0,0) --(0,0,2);
    \draw (0,2,2) --  (0,0,2) --  (2,0,2);
    \draw[dashed] (-2,1,0) --(0,0,2);
    \draw (-2,1,0) --  (0,2,2);
    \draw[thick,->] (-1-2*0.2,1.5-1*0.2,1-2*0.2) -- (-1+2*0.2,1.5+1*0.2,1+2*0.2);
    \draw[thick,->] (-1-2*0.2,1.5-1*0.2,1-2*0.2) -- (-1-2*0.2 - 1/3,1.5-1*0.2 +2/3 ,1-2*0.2);
    \draw[thick,->] (-1-2*0.2,1.5-1*0.2,1-2*0.2) -- (-1-2*0.2 - 4/8,1.5-1*0.2 -2/8 ,1-2*0.2 +5/8);
    \draw(-1+2*0.01,1.5-1*0,1-2*0.05) node{$t_e$};
    \draw(-1-2*0.4,1.5-1*0,1-2*0.32) node{$\nu_{e,1}$};
    \draw(-1-2*0.4,1.5-1*0,1-2*0.02) node{$\nu_{e,2}$};
    \draw[thick,->] (0.6,0.6,2) -- (0.6,0.6,2+0.65);
    \draw(1,1,2.8) node{$\nu_F$};
    \draw[thick,->] (0.6,0.6,2) -- (0.6+0.7,0.6,2);
    \draw(1.65,1,2) node{$t_{F,1}$};
    \draw[thick,->] (0.6,0.6,2) -- (0.6,0.6+0.9,2);
    \draw(1.12,1.8,2.07) node{$t_{F,2}$};
  \end{tikzpicture}     
  \caption{Normal vectors and tangential vectors for face $F$ and edge $e$ of a polyhedron}\label{fig:normaltangentialvectors}
\end{figure}
Define the surface gradient on $F$ as
\begin{equation*}
\nabla_{F}v:=\nabla v-\sum_{i=1}^r\frac{\partial v}{\partial\nu_{F,i}}\nu_{F,i}=\sum_{i=1}^{n-r}\frac{\partial v}{\partial t_{F,i}}t_{F,i},
\end{equation*}
namely the projection of $\nabla v$ to the face $F$, which is independent of the choice of the normal vectors. And denote by $\div_{F}$ the corresponding surface divergence.
For any $\delta\in\mathcal{F}_h^n$ and $i=1,\cdots, n$, let $\nu_{\delta,i}:=e_i=(0,\cdots,0,1,0,\cdots,0)^{\intercal}$ be the $n$-tuple with all entries equal to $0$, except the $i$th, which is $1$.
For any $F\in\mathcal F^{r}_h$, $\alpha\in A_{r}$ and $\beta\in A_{n-r}$ with $r=1,\cdots, n$, set
$$
\nu_{F}^{\alpha}:=\nu_{F,1}^{\alpha_1}\otimes\cdots\otimes\nu_{F,r}^{\alpha_r},\quad t_{F}^{\beta}:=t_{F,1}^{\beta_1}\otimes\cdots\otimes t_{F,n-r}^{\beta_{n-r}},
$$
\[
\frac{\partial^{|\alpha|}v}{\partial\nu_{F}^{\alpha}}:=\frac{\partial^{|\alpha|}v}{\partial\nu_{F, 1}^{\alpha_1}\cdots\partial\nu_{F, r}^{\alpha_{r}}},\quad \frac{\partial^{|\beta|}v}{\partial t_{F}^{\beta}}:=\frac{\partial^{|\beta|}v}{\partial t_{F, 1}^{\beta_1}\cdots\partial t_{F, n-r}^{\beta_{n-r}}},
\]
where $\nu_{F,i}^{\alpha_i}:=\underbrace{\nu_{F,i}\otimes\cdots\otimes\nu_{F,i}}_{\alpha_i}$ and $t_{F,i}^{\beta_i}:=\underbrace{t_{F,i}\otimes\cdots\otimes t_{F,i}}_{\beta_i}$.
For any $e\in\mathcal F^s(F)$ with $1\leq s< n-r$, let $\nu_{F,e,1}, \cdots, \nu_{F,e,s}$ be its
mutually perpendicular unit normal vectors paralleling to $F$, and abbreviate $\nu_{F,e,1}$ as $\nu_{F,e}$ when $s=1$. And for any $\delta\in\mathcal F^{n-r}(F)$, let $\nu_{F,\delta,i}:=t_{F,i}$ for $i=1,\cdots, n-r$.
Set
$$
\nu_{F,e}^{\beta}:=\nu_{F,e,1}^{\beta_1}\otimes\cdots\otimes\nu_{F,e,s}^{\beta_s},\quad
\frac{\partial^{|\beta|}v}{\partial\nu_{F,e}^{\beta}}:=\frac{\partial^{|\beta|}v}{\partial\nu_{F, e, 1}^{\beta_1}\cdots\partial\nu_{F, e, s}^{\beta_{s}}}
$$
for any $\beta\in A_{s}$ with $s=1,\cdots, n-r$.
For any $K\in\mathcal T_h$, $\delta\in \mathcal F^n(K)$, and any function $v$ defined on $K$,
we will rewrite $v(\boldsymbol x_{\delta})$ as $v(\delta)$ for simplicity, where $\boldsymbol x_{\delta}$ is the position of the point $\delta$.


\subsection{Mesh conditions}
We impose the following conditions on the mesh $\mathcal T_h$.
\begin{itemize}
 \item[(A1)] Each element $K\in \mathcal T_h$ and each face $F\in \mathcal F_h^r$ for $1\leq r\leq n-1$ is star-shaped with a uniformly bounded chunkiness parameter. For a domain $D$, the chunkiness parameter $\gamma_D:=h_D/\rho_D$, where $\rho_D$ is the radius of the largest ball contained in $D$.

 \item[(A2)] There exists a real number $\eta>0$ such that for each $K\in \mathcal T_h$, $h_K\leq\eta h_F$ for all $F\in\mathcal F^r(K)$ with $r=1,\cdots, n-1$.
\end{itemize}

Throughout this paper, we also use
``$\lesssim\cdots $" to mean that ``$\leq C\cdots$", where
$C$ is a generic positive constant independent of mesh size $h$, but may depend on the chunkiness parameter of the polytope, constant $\eta$, the degree of polynomials $k$, the order of differentiation $m$, and the dimension of space $n$,
which may take different values at different appearances. And $A\eqsim B$ means $A\lesssim B$ and $B\lesssim A$.
Hereafter, we always assume $k\geq m$.

For a star-shaped domain $D$, it holds the multiplicative trace inequality (cf. \cite[Theorem 1.5.1.10]{Grisvard1985})
\begin{equation}\label{L2tracemulti}
\|v\|_{0,\partial D}^2 \lesssim h_D^{-1}\|v\|_{0,D}(\|v\|_{0,D} + h_D |v|_{1,D}) \quad \forall~v\in H^1(D).
\end{equation}
This implies the trace inequality (cf. \cite[(2.18)]{BrennerSung2018})
\begin{equation}\label{L2trace}
\|v\|_{0,\partial D}^2 \lesssim h_D^{-1}\|v\|_{0,D}^2 + h_D |v|_{1,D}^2 \quad \forall~v\in H^1(D).
\end{equation}
When $D$ is a set of a finite number of points, the notation $\|v\|_{0, D}$ means $\|v\|_{L^{\infty}(D)}$.
We also have the Poincar\'e-Friedrichs inequality \cite[(2.15)]{BrennerSung2018}
\begin{equation}\label{eq:Poincare-Friedrichs2}
\|v\|_{0,D}\lesssim h_D|v|_{1,D}\quad \forall~v\in H_0^1(D),
\end{equation}
and 
the inverse inequality for polynomials \cite[Lemma~10]{Huang2020}
\begin{equation}\label{eq:polyinverse}
 \| q \|_{0,D} \lesssim h_D^{-i}\| q \|_{-i,D} \quad \forall~q\in \mathbb P_{\ell}(D)
\end{equation}
for any non-negative integers $\ell$ and $i$.
As a result of \eqref{eq:polyinverse}, the Bramle-Hilbert lemma \cite[Lemma 4.3.8]{BrennerScott2008} and \eqref{L2trace},
it holds the estimate of the $L^2$-orthogonal projection
\begin{equation}\label{eq:l2projectionestimate}
h_D^{i}|v-Q_k^Dv|_{i,D}+ h_D^{1/2}\|v-Q_k^Dv\|_{0,\partial D}\lesssim h_D^{j}|v|_{j,D}\quad\forall~v\in H^j(D) 
\end{equation}
for $0\leq i\leq j\leq k+1$ with $i,j,k$ being non-negative integers.
The hidden constants in \eqref{L2tracemulti}-\eqref{eq:l2projectionestimate} depend on the chunkiness parameter $\eta$ and the spatial dimension $n$.

\section{$H^m$-Conforming Virtual Elements}\label{sec:vem}

We will construct $H^m$-conforming virtual elements $(K, \mathcal N_k^m(K), V_k^m(K))$ for any integers $m,n\geq1$, $k\geq m$ and $n$-dimensional polytope $K\subset\mathbb R^n$ by gluing conforming virtual elements on faces recursively.

We first list a Green's identity for later uses.

\begin{lemma}
For any $v\in H^m(K)$ and $q\in H^{2m}(K)$,
\begin{equation}\label{eq:greenidentity}
(\nabla^mv, \nabla^mq)_K=(v, (-\Delta)^mq)_K+\sum_{i=0}^{m-1}(\nabla^{i}v, \nabla^{i}(-\Delta)^{m-i-1}\partial_{\nu}q)_{\partial K},
\end{equation}
where $\partial_{\nu}q|_F:=\frac{\partial q}{\partial\nu_{F}}$ for each face $F\in\mathcal F^1(K)$.
\end{lemma}
\begin{proof}
For $i=0,1,\cdots, m-1$, applying the integration by parts, it follows
$$
(\nabla^{i+1}v, \nabla^{i+1}(-\Delta)^{m-i-1}q)_K = (\nabla^{i}v, \nabla^{i}(-\Delta)^{m-i}q)_K +(\nabla^{i}v, \nabla^{i}(-\Delta)^{m-i-1}\partial_{\nu}q)_{\partial K}.
$$
Thus \eqref{eq:greenidentity} holds from the sum of the last identity from $i = 0$ to $m-1$.
\end{proof}

\begin{lemma}
Let $F\in\mathcal{F}^r(K)$ with $1\leq r\leq n-1$, and integer $j>0$. It holds for any smooth function $v$ that
\begin{equation}\label{eq:20210413}
\nabla^jv=\sum_{\alpha\in A_r, \beta\in A_{n-r}\atop|\alpha|+|\beta|= j}\frac{j!}{\alpha!\beta!}\sym(\nu_F^{\alpha}\otimes t_F^{\beta})\frac{\partial^{j}v}{\partial t_F^{\beta}\partial \nu_F^{\alpha}}.
\end{equation}
\end{lemma}
\begin{proof}
Recalling that $\nabla^jv$ is a symmetric $j$-tensor,
apparently it follows
\begin{equation}\label{eq:20210403}
\nabla^jv=\sym(\nabla^jv)=\sym\bigg(\sum_{i=1}^r\nu_{F,i}\frac{\partial }{\partial\nu_{F,i}}+\sum_{i=1}^{n-r}t_{F,i}\frac{\partial }{\partial t_{F,i}}\bigg)^jv.
\end{equation}
We conclude \eqref{eq:20210413} by applying the multinomial theorem to \eqref{eq:20210403}.
\end{proof}

\subsection{$H^m$-conforming virtual elements in one dimension}
We start from one dimension, i.e. $n=1$. Now the polytope $K$ is an interval. 
The DoFs $\mathcal N_k^m(K)$ are chosen as
\begin{align}
h_K^jv^{(j)}(\delta) & \quad\forall~\delta\in\mathcal F^{1}(K), \;j=0,1,\cdots,m-1, \label{eq:cfmvem1ddof1}\\
\frac{1}{|K|}(v, q)_K & \quad\forall~q\in\mathbb P_{k-2m}(K), \label{eq:cfmvem1ddof2}
\end{align}
where $v^{(j)}$ is the $j$th order derivative of $v$.
And take the space of shape functions
$$
V_k^m(K):=\big\{ v\in H^m(K): v^{(2m)}\in \mathbb P_{k-2m}(K)\big\}.
$$
Clearly we have
$$
V_k^m(K)=\begin{cases}
\mathbb P_{k}(K), & k\geq 2m-1,\\
\mathbb P_{2m-1}(K), & k<2m-1.
\end{cases}
$$
Hence the $H^m$-conforming virtual element of degree $k$ in one dimension is exactly the $C^{m-1}$-continuous finite element, whose shape functions are polynomials of degree $\max\{k, 2m-1\}$.
And the $H^m$-conforming virtual elements $(K, \mathcal N_k^m(K), V_k^m(K))$ coincide with the nonconforming ones in \cite[Remark~1]{Huang2020}.

\subsection{$H^m$-conforming virtual elements in two dimensions}
Then we consider the construction of the $H^m$-conforming virtual elements in two dimensions, i.e. $n=2$, where the polytope $K$ is a polygon.
The $H^m$-conforming virtual elements $(K, \mathcal N_k^m(K), V_k^m(K))$ in two dimensions have been designed in \cite{BeiraoManzini2014,AntoniettiManziniVerani2020,AntoniettiManziniScacchiVerani2021,BrezziMarini2013}.
Here we review them to motivate the construction of $H^m$-conforming virtual elements in higher dimensions.

The space of shape functions in the virtual elements is defined through local partial differential equations \cite{BeiraodaVeigaBrezziCangianiManziniEtAl2013,AntoniettiManziniVerani2020}.
To ensure the $L^2$ projection $Q_k^{K}v$ is computable for any virtual element function $v$  by using the DoFs, following the idea in \cite{AhmadAlsaediBrezziMariniEtAl2013} 
we first define a preliminary virtual element space with the help of the conforming virtual elements in one dimension
\begin{align*}
\widetilde{V}_k^m(K):=\Big\{ v\in H^m(K):& (-\Delta)^mv\in \mathbb P_{k}(K), \\
&\frac{\partial^jv}{\partial \nu_{e}^j}\Big|_e\in V_{k-j}^{m-j}(e)\quad\forall~e\in\mathcal F^{1}(K),\; j=0,1,\cdots, m-1\Big\}.    
\end{align*}
Clearly $\mathbb P_k(K)\subseteq\widetilde{V}_k^m(K)$. On the other hand, $\frac{\partial^{m-1}v}{\partial \nu_e^{m-1}}\Big|_e\in V_{k-(m-1)}^{1}(e)=\mathbb P_{k-m+1}(e)$, thus $\mathbb P_{k+1}(K)\not\subseteq\widetilde{V}_k^m(K)$.

\begin{lemma}\label{lem:continuous2d}
For any $v\in\widetilde{V}_k^m(K)$, $\nabla^jv$ is continuous on $\partial K$, and $(\nabla^jv)|_{\partial K}\in H^1(\partial K;\mathbb S_{2}(j))$ for $j=0,1,\cdots, m-1$.
\end{lemma}
\begin{proof}
On each edge $e\in\mathcal F^1(K)$, it holds from \eqref{eq:20210413} that
\begin{equation}\label{eq:202104072d}
\nabla^jv=\sum_{\ell=0}^j\frac{j!}{\ell!(j-\ell)!}\sym(t_e^{\ell}\otimes\nu_e^{j-\ell})\frac{\partial^{\ell}}{\partial t_e^{\ell}}\left(\frac{\partial^{j-\ell}v}{\partial\nu_e^{j-\ell}}\right).
\end{equation}
By the definition of $\widetilde{V}_k^m(K)$, $\frac{\partial^{j-\ell}v}{\partial\nu_e^{j-\ell}}\Big|_e\in V_{k-j+\ell}^{m-j+\ell}(e)=\mathbb P_{\max\{k-j+\ell, 2m-1-2j+2\ell\}}(e)$ is a polynomial. Hence it follows from \eqref{eq:202104072d} that $\nabla^jv|_e\in\mathbb P_{\max\{k-j, 2m-1-j\}}(e,\mathbb S_2(j))$ is a tensor with components being polynomials. Finally we acquire from the fact $v\in H^m(K)$ that $\nabla^jv$ is continuous on $\partial K$ for $j=0,1,\cdots, m-1$ (cf. comments after Theorem~1.5.2.3 in \cite{Grisvard1985}), which also means $(\nabla^jv)|_{\partial K}\in H^1(\partial K;\mathbb S_{2}(j))$.
\end{proof}

In the definition of $\widetilde{V}_k^m(K)$, $\frac{\partial^jv}{\partial \nu_e^j}\Big|_e\in V_{k-j}^{m-j}(e)$ and the DoFs \eqref{eq:cfmvem1ddof1}-\eqref{eq:cfmvem1ddof2} in one dimension motivate us that the DoFs in two dimensions should cover the following ones
\begin{align*}
h_e^i\frac{\partial^i}{\partial t_e^i}\left(\frac{\partial^jv}{\partial \nu_e^j}\Big|_e\right)(\delta) & \quad\forall~\delta\in\mathcal F^{1}(e), \;i=0,1,\cdots,m-j-1, \\
\frac{1}{|e|}(\frac{\partial^jv}{\partial \nu_e^j}, q)_e & \quad\forall~q\in\mathbb P_{k-j-2(m-j)}(e). 
\end{align*}
Noting that $\nabla^jv$ is continuous on $\partial K$ for $j=0,1,\cdots, m-1$, we propose the following DoFs $\mathcal N_k^m(K)$ for the $H^m$-conforming virtual elements in two dimensions
\begin{align}
h_K^j\nabla^{j}v(\delta) & \quad\forall~\delta\in\mathcal F^{2}(K), \;j=0,1,\cdots,m-1, \label{eq:cfmvem2ddof1}\\
\frac{1}{|e|^{1-j}}(\frac{\partial^jv}{\partial \nu_e^j}, q)_e & \quad\forall~q\in\mathbb P_{k-2m+j}(e), e\in\mathcal F^{1}(K), \;j=0,1,\cdots,m-1, \label{eq:cfmvem2ddof2}\\
\frac{1}{|K|}(v, q)_K & \quad\forall~q\in\mathbb P_{k-2m}(K). \label{eq:cfmvem2ddof3}
\end{align}

To define the space of shape functions $V_k^m(K)$, we also need a local $H^m$ projection operator $\Pi_k^K: H^m(K)\to\mathbb P_k(K)$: given $v\in H^m(K)$, let $\Pi_k^Kv\in\mathbb P_k(K)$ be the solution of the problem 
\begin{align}
(\nabla^m\Pi_k^Kv, \nabla^mq)_K&=(\nabla^mv, \nabla^mq)_K\quad  \forall~q\in \mathbb P_k(K),\label{eq:projlocal2d1}\\
\sum_{\delta\in\mathcal F^{2}(K)}(\nabla^{j}\Pi_k^Kv)(\delta)&=\sum_{\delta\in\mathcal F^{2}(K)}(\nabla^{j}v)(\delta), \quad j=0,1,\cdots, m-1.\label{eq:projlocal2d2}
\end{align}
The number of equations in~\eqref{eq:projlocal2d2} is
\[
\sum_{j=0}^{m-1}(j+1)=\frac{1}{2}m(m+1)=\dim\mathbb P_{m-1}(K).
\]
Applying the argument in \cite[Section~3.3 and Lemma~3.5]{ChenHuang2020}, the local problem \eqref{eq:projlocal2d1}-~\eqref{eq:projlocal2d2} is well-posed, and it holds
\begin{equation}\label{eq:localproj2dprop1}
\Pi_k^Kq=q \quad\forall~q\in\mathbb P_k(K).    
\end{equation}
Notice that $\nabla^jv|_e\in \mathbb P_{\max\{k-j, 2m-1-j\}}(e,\mathbb S_2(j))$ is a tensor with polynomial components for each $e\in\mathcal F^1(K)$ and $j=0,1,\cdots, m-1$, which is computable using the DoFs~\eqref{eq:cfmvem2ddof1}-\eqref{eq:cfmvem2ddof2} for any $v\in \widetilde{V}_k^m(K)$.
Then we get from the Green's identity \eqref{eq:greenidentity} that
the projection $\Pi_k^Kv$ is computable using only the DoFs~\eqref{eq:cfmvem2ddof1}-\eqref{eq:cfmvem2ddof3} for any $v\in \widetilde{V}_k^m(K)$.

Following the ideas in \cite{AhmadAlsaediBrezziMariniEtAl2013,ChenHuang2020}, define the space of shape functions
$$
V_k^m(K):=\{v\in \widetilde{V}_k^m(K): (v-\Pi_k^Kv, q)_K=0\quad\forall~q\in\mathbb P_{k-2m}^{\perp}(K)\}.
$$
Due to \eqref{eq:localproj2dprop1}, it holds $\mathbb P_k(K)\subseteq V_k^m(K)$ and $\mathbb P_{k+1}(K)\not\subseteq V_k^m(K)$.
Therefore we arrive at the $H^m$-conforming virtual elements $(K, \mathcal N_k^m(K), V_k^m(K))$ in two dimensions.
The uni-solvence of $(K, \mathcal N_k^m(K), V_k^m(K))$ will be covered in the arbitrary dimension in Subsections~\ref{subsec:dataspacetrace} and \ref{subsec:vemunisolvence}.

For any $v\in V_k^m(K)$, since $(v-\Pi_k^Kv, q-Q_{k-2m}^Kq)_K=0$ for each $q\in\mathbb P_k(K)$, we have
$$
(Q_k^K-Q_{k-2m}^K)(v-\Pi_k^Kv)=Q_k^K(I-Q_{k-2m}^K)(v-\Pi_k^Kv)=0.
$$
This yields
$$
Q_k^Kv= \Pi_k^Kv + Q_{k-2m}^Kv-Q_{k-2m}^K\Pi_k^Kv.
$$
Hence $Q_k^Kv$ is computable using only the DoFs~\eqref{eq:cfmvem2ddof1}-\eqref{eq:cfmvem2ddof3} for any $v\in V_k^m(K)$.
This combined with the integration by parts implies that $Q_{k+j}^K(\nabla^jv)$ is computable using only the DoFs~\eqref{eq:cfmvem2ddof1}-\eqref{eq:cfmvem2ddof3} for any $v\in V_k^m(K)$ and $j=1, \cdots, m$.


\subsection{$H^m$-conforming virtual elements in three dimensions}
Next we construct the $H^m$-conforming virtual elements for $k\geq m$ and $m\geq1$ in three dimensions.
Several $H^2$-conforming virtual elements in three dimensions are devised in~\cite{BeiraodaVeigaDassiRusso2020,ChenHuang2022,BrennerSung2019}.

Let polyhedron $K\subset\mathbb R^3$.
Similarly as the two dimensions, 
we first define a preliminary virtual element space
\begin{align*}
\widetilde{V}_k^m(K):=\big\{ &v\in H^m(K): (-\Delta)^mv\in \mathbb P_{k}(K), \\
& (\nabla^jv)|_{\mathcal S_K^r}\in H^1(\mathcal S_K^r;\mathbb S_{3}(j)) \;\textrm{ for } j=0,1,\cdots,m-1 \textrm{ and } r=1,2, \\
&\frac{\partial^{j}v}{\partial \nu_F^{j}}\Big|_F\in V_{k-j}^{m-j}(F)\;\textrm{ for } F\in\mathcal F^{1}(K), j=0,1,\cdots,m-1,\\
&\hskip-0.4cm\frac{\partial^{j}v}{\partial \nu_{e,1}^{i}\partial \nu_{e,2}^{j-i}}\Big|_e\in V_{k-j}^{m-j}(e)\;\textrm{ for } e\in\mathcal F^{2}(K), 0\leq i\leq j,\, j=0,1,\cdots, m-1\}.    
\end{align*}
The requirement $(\nabla^jv)|_{\mathcal S_K^r}\in H^1(\mathcal S_K^r;\mathbb S_{3}(j))$ in the definition of $\widetilde{V}_k^m(K)$ is motivated by Lemma~\ref{lem:continuous2d}.
Since $\mathbb P_{k-j}(F)\in V_{k-j}^{m-j}(F)$ and $\mathbb P_{k-j}(e)\in V_{k-j}^{m-j}(e)$, we have $\mathbb P_k(K)\subseteq\widetilde{V}_k^m(K)$.
Take $v\in\widetilde{V}_k^m(K)$.
By the definition of $\widetilde{V}_k^m(K)$, $\frac{\partial^{j-\ell}v}{\partial \nu_{e,1}^{i}\partial \nu_{e,2}^{j-\ell-i}}\Big|_e\in V_{k-j+\ell}^{m-j+\ell}(e)=\mathbb P_{\max\{k-j+\ell, 2m-1-2j+2\ell\}}(e)$ is a polynomial for each edge $e\in\mathcal F^{2}(K)$. It follows from \eqref{eq:20210413} that $\nabla^jv|_e\in\mathbb P_{\max\{k-j, 2m-1-j\}}(e,\mathbb S_3(j))$ is a tensor with components being polynomials. By $(\nabla^jv)|_{\mathcal S_K^{2}}\in H^1(\mathcal S_K^{2};\mathbb S_{3}(j))$, 
$\nabla^jv$ is continuous on the one-dimensional skeleton $\mathcal S_K^{2}$.
For $F\in\mathcal{F}^1(K)$ and $e\in\mathcal{F}^2(K)$, applying \eqref{eq:20210413}, we have $\nabla^jv|_F\in H^{m-j}(F; \mathbb S_{3}(j))$ and $\nabla^jv|_e\in H^{m-j}(e; \mathbb S_{3}(j))$.

Inspired by $\frac{\partial^{j}v}{\partial \nu_{e,1}^{i}\partial \nu_{e,2}^{j-i}}\Big|_e\in V_{k-j}^{m-j}(e)$, $\frac{\partial^{j}v}{\partial \nu_F^{j}}\Big|_F\in V_{k-j}^{m-j}(F)$, the DoFs~\eqref{eq:cfmvem1ddof1}-\eqref{eq:cfmvem1ddof2} and the DoFs~\eqref{eq:cfmvem2ddof1}-\eqref{eq:cfmvem2ddof3}, 
we propose the following DoFs $\mathcal N_k^m(K)$ for the $H^m$-conforming virtual elements in three dimensions
\begin{align}
h_K^j\nabla^{j}v(\delta) & \quad\forall~\delta\in\mathcal F^{3}(K), \;j=0,1,\cdots,m-1, \label{eq:cfmvem3ddof1}\\
\frac{h_K^{j}}{|e|}(\frac{\partial^{j}v}{\partial \nu_{e,1}^{i}\partial \nu_{e,2}^{j-i}}, q)_e & \quad\forall~q\in\mathbb P_{k-2m+j}(e), e\in\mathcal F^{2}(K), \label{eq:cfmvem3ddof2}\\
&\quad \quad 0\leq i\leq j,\, j=0,1,\cdots, m-1, \notag\\
\frac{h_K^{j}}{|F|}(\frac{\partial^{j}v}{\partial \nu_F^{j}}, q)_F & \quad\forall~q\in\mathbb P_{k-2m+j}(F), F\in\mathcal F^{1}(K), \;j=0,1,\cdots,m-1, \label{eq:cfmvem3ddof3}\\
\frac{1}{|K|}(v, q)_K & \quad\forall~q\in\mathbb P_{k-2m}(K). \label{eq:cfmvem3ddof4}
\end{align}

To define the space of shape functions $V_k^m(K)$, we introduce a local $H^m$-projector $\Pi_k^K: H^m(K)\to\mathbb P_k(K)$: given $v\in H^m(K)$, let $\Pi_k^Kv\in\mathbb P_k(K)$ be the solution of the problem 
\begin{align}
(\nabla^m\Pi_k^Kv, \nabla^mq)_K&=(\nabla^mv, \nabla^mq)_K\quad  \forall~q\in \mathbb P_k(K),\label{eq:projlocal3d1}\\
\sum_{\delta\in\mathcal F^{3}(K)}(\nabla^{j}\Pi_k^Kv)(\delta)&=\sum_{\delta\in\mathcal F^{3}(K)}(\nabla^{j}v)(\delta), \quad j=0,1,\cdots, m-1.\label{eq:projlocal3d2}
\end{align}
The number of equations in~\eqref{eq:projlocal3d2} is
\[
\sum_{j=0}^{m-1}C_{j+2}^{2}=C_{m+2}^{3}=\dim\mathbb P_{m-1}(K).
\]
The problem \eqref{eq:projlocal3d1}-\eqref{eq:projlocal3d2} is well-posed, and it holds the identity
\begin{equation}\label{eq:localproj3dprop1}
\Pi_k^Kq=q \quad\forall~q\in\mathbb P_k(K).    
\end{equation}
For $v\in\widetilde{V}_k^m(K)$, $\nabla^jv|_e$ on edge $e\in\mathcal{F}^2(K)$ is clearly computable by using the DoFs \eqref{eq:cfmvem3ddof1}-\eqref{eq:cfmvem3ddof2} for $j=0,1,\cdots, m-1$, since $\nabla^jv|_e$ is a tensor-valued polynomial.
By \eqref{eq:20210413}, $Q_{k-j}^F(\nabla^jv)$ is computable by using the DoFs \eqref{eq:cfmvem3ddof1}-\eqref{eq:cfmvem3ddof3} for $F\in\mathcal{F}^1(K)$ and $j=0,1,\cdots, m-1$. 
Therefore it follows from \eqref{eq:greenidentity} that the projection $\Pi_k^Kv$ is computable using only the DoFs~\eqref{eq:cfmvem3ddof1}-\eqref{eq:cfmvem3ddof4} for any $v\in \widetilde{V}_k^m(K)$.


Define the space of shape functions
$$
V_k^m(K):=\{v\in \widetilde{V}_k^m(K): (v-\Pi_k^Kv, q)_K=0\quad\forall~q\in\mathbb P_{k-2m}^{\perp}(K)\}.
$$
Due to \eqref{eq:localproj3dprop1}, it holds $\mathbb P_k(K)\subseteq V_k^m(K)$.
We finish the construction of the $H^m$-conforming virtual elements $(K, \mathcal N_k^m(K), V_k^m(K))$ in three dimensions.

\subsection{$H^m$-conforming virtual elements in arbitrary dimension}
Now we construct the $H^m$-conforming virtual elements for $k\geq m$ and $m\geq1$ in arbitrary dimension recursively.

Let polytope $K\subset\mathbb R^n$ with $n\geq 2$.
Assume $H^{\ell}$-conforming virtual elements $(F, \mathcal N_{k_{\ell}}^{\ell}(F), V_{k_{\ell}}^{\ell}(F))$ for $\ell=1,\cdots, m$ and $k_{\ell}\geq\ell$ have been constructed for each $F\in\mathcal F^r(K)$ with $r=1,\cdots, n-1$. The DoFs $\mathcal N_{k_{\ell}}^{\ell}(F)$ are given by
\begin{align}
h_F^j\nabla_F^{j}v(\delta) & \quad\forall~\delta\in\mathcal F^{n-r}(F), \;j=0,1,\cdots,\ell-1, \label{eq:cfmvemfacesdof1}\\
\frac{h_F^{|\alpha|}}{|e|}(\frac{\partial^{|\alpha|}v}{\partial \nu_{F,e}^{\alpha}}, q)_e & \quad\forall~q\in\mathbb P_{k_{\ell}-2\ell+|\alpha|}(e), e\in\mathcal F^{s}(F), \;s=1,\cdots,n-r-1, \label{eq:cfmvemfacesdof2}\\
&\quad \quad \alpha\in A_{s}, \textrm{ and } |\alpha|\leq \ell-1, \notag\\
\frac{1}{|F|}(v, q)_F & \quad\forall~q\in\mathbb P_{k_{\ell}-2\ell}(F). \label{eq:cfmvemfacesdof3}
\end{align}
And assume 
\begin{enumerate}[(i)]
\item $\mathbb P_{k_{\ell}}(F)\subseteq V_{k_{\ell}}^{\ell}(F)\subset H^{\ell}(F)$;
\item for any $v\in V_{k_{\ell}}^{\ell}(F)$, we have $(\nabla_F^jv)|_{\mathcal S_F^s}\in H^{1}(\mathcal S_F^s; \mathbb S_{n-r}(j))$ and $(\nabla_F^jv)|_{e}\in H^{\ell-j}(e; \mathbb S_{n-r}(j))$ for $e\in\mathcal F^s(F)$, $j=0,1,\cdots,\ell-1$ and $s=1,\cdots, n-r-1$;
\item for any $v\in V_{k_{\ell}}^{\ell}(F)$, $\frac{\partial^{|\beta|}v}{\partial \nu_{F,e}^{\beta}}\Big|_{e}\in V_{k_{\ell}-|\beta|}^{\ell-|\beta|}(e)$ for each $e\in\mathcal F^s(F)$, $\beta\in A_s$, $|\beta|\leq\ell-1$ and $s=1,\cdots, n-r-1$;
\item $Q_{k_{\ell}}^F(\nabla_F^jv)$ is computable using only the DoFs~\eqref{eq:cfmvemfacesdof1}-\eqref{eq:cfmvemfacesdof3} for any $v\in V_{k_{\ell}}^{\ell}(F)$ and $j=0,1,\cdots, \ell$.
\end{enumerate}
The assumption (ii) is inspired by Lemma~\ref{lem:continuous2d}.

First define a preliminary virtual element space
\begin{align*}
\widetilde{V}_k^m(K):=\big\{ v\in H^m(K): & (-\Delta)^mv\in \mathbb P_{k}(K), \\
& (\nabla^jv)|_{\mathcal S_K^r}\in H^1(\mathcal S_K^r;\mathbb S_{n}(j)) \textrm{ for } j=0,1,\cdots,m-1, \\
&\frac{\partial^{|\alpha|}v}{\partial \nu_F^{\alpha}}\Big|_F\in V_{k-|\alpha|}^{m-|\alpha|}(F)\quad\forall~F\in\mathcal F^{r}(K),\\
&\qquad\quad\;\;\, r=1,\cdots, n-1, \alpha\in A_r, \textrm{ and } |\alpha|\leq m-1\big\}.    
\end{align*}
By the assumption (i), we have $\mathbb P_k(K)\subseteq\widetilde{V}_k^m(K)$.
Take $v\in\widetilde{V}_k^m(K)$.
Applying the same argument as in Lemma~\ref{lem:continuous2d}, 
$\nabla^jv|_e\in\mathbb P_{\max\{k-j, 2m-1-j\}}(e,\mathbb S_n(j))$ for each edge $e\in\mathcal F^{n-1}(K)$ and $j=0,1,\cdots, m-1$, then it follows from $(\nabla^jv)|_{\mathcal S_K^{n-1}}\in H^1(\mathcal S_K^{n-1};\mathbb S_{n}(j))$ that
$\nabla^jv$ is continuous on the one-dimensional skeleton $\mathcal S_K^{n-1}$.
For any $F\in\mathcal{F}^r(K)$ with $1\leq r\leq n-1$, 
we get from the definition of $\widetilde{V}_k^m(K)$ and~\eqref{eq:20210413} that $\nabla^jv|_F\in H^{m-j}(F; \mathbb S_{n}(j))$.

Inspired by $\frac{\partial^{|\alpha|}v}{\partial \nu_F^{\alpha}}\Big|_F\in V_{k-|\alpha|}^{m-|\alpha|}(F)$ and the DoFs~\eqref{eq:cfmvemfacesdof1}-\eqref{eq:cfmvemfacesdof3}, 
we propose the following degrees of freedom (DoFs) $\mathcal N_k^m(K)$ for the $H^m$-conforming virtual elements in arbitrary dimension
\begin{align}
h_K^j\nabla^{j}v(\delta) & \quad\forall~\delta\in\mathcal F^{n}(K), \;j=0,1,\cdots,m-1, \label{eq:cfmvemdof1}\\
\frac{h_K^{|\alpha|}}{|F|}(\frac{\partial^{|\alpha|}v}{\partial \nu_F^{\alpha}}, q)_F & \quad\forall~q\in\mathbb P_{k-2m+|\alpha|}(F), F\in\mathcal F^{r}(K), \;r=1,\cdots,n-1, \label{eq:cfmvemdof2}\\
&\quad \quad \alpha\in A_r, \textrm{ and } |\alpha|\leq m-1, \notag\\
\frac{1}{|K|}(v, q)_K & \quad\forall~q\in\mathbb P_{k-2m}(K). \label{eq:cfmvemdof3}
\end{align}

To define the space of shape functions $V_k^m(K)$, we introduce a local $H^m$-projector $\Pi_k^K: H^m(K)\to\mathbb P_k(K)$: given $v\in H^m(K)$, let $\Pi_k^Kv\in\mathbb P_k(K)$ be the solution of the problem 
\begin{align}
(\nabla^m\Pi_k^Kv, \nabla^mq)_K&=(\nabla^mv, \nabla^mq)_K\quad  \forall~q\in \mathbb P_k(K),\label{eq:projlocal1}\\
\sum_{\delta\in\mathcal F^{n}(K)}(\nabla^{j}\Pi_k^Kv)(\delta)&=\sum_{\delta\in\mathcal F^{n}(K)}(\nabla^{j}v)(\delta), \quad j=0,1,\cdots, m-1.\label{eq:projlocal2}
\end{align}
The number of equations in~\eqref{eq:projlocal2} is
\[
\sum_{j=0}^{m-1}C_{n-1+j}^{n-1}=C_{n+m-1}^{n}=\dim\mathbb P_{m-1}(K).
\]
We refer to \cite[Section~3.3 and Lemma~3.5]{ChenHuang2020} for the well-posedness of \eqref{eq:projlocal1}-\eqref{eq:projlocal2}, and the identity
\begin{equation}\label{eq:localprojprop1}
\Pi_k^Kq=q \quad\forall~q\in\mathbb P_k(K).    
\end{equation}
By the assumption (iv) of conforming virtual elements on faces, $Q_{k-|\alpha|}^F(\nabla_F^{\ell}\frac{\partial^{|\alpha|}}{\nu_F^{\alpha}}v)$ is computable by using the DoFs \eqref{eq:cfmvemdof1}-\eqref{eq:cfmvemdof2} for any $v\in\widetilde{V}_k^m(K)$, $F\in\mathcal F^r(K)$, $\alpha\in A_r$, $|\alpha|\leq m-1$, $r=1,\cdots, n-1$, and $\ell=0,\cdots,m-|\alpha|$. This together with \eqref{eq:20210413} implies $Q_{k-j}^F(\nabla^jv)$ is computable by using the DoFs \eqref{eq:cfmvemdof1}-\eqref{eq:cfmvemdof2} for any $v\in\widetilde{V}_k^m(K)$. Therefore it follows from \eqref{eq:greenidentity} that the projection $\Pi_k^Kv$ is computable using only the DoFs~\eqref{eq:cfmvemdof1}-\eqref{eq:cfmvemdof3} for any $v\in \widetilde{V}_k^m(K)$.



Following the ideas in \cite{AhmadAlsaediBrezziMariniEtAl2013,ChenHuang2020}, define the space of shape functions
$$
V_k^m(K):=\{v\in \widetilde{V}_k^m(K): (v-\Pi_k^Kv, q)_K=0\quad\forall~q\in\mathbb P_{k-2m}^{\perp}(K)\}.
$$
Due to \eqref{eq:localprojprop1}, it holds $\mathbb P_k(K)\subseteq V_k^m(K)$.
Finally we finish the construction of the $H^m$-conforming virtual elements $(K, \mathcal N_k^m(K), V_k^m(K))$ in arbitrary dimension.

\subsection{Data spaces and trace}\label{subsec:dataspacetrace}
From now on in this section we will show that the DoFs~\eqref{eq:cfmvemdof1}-\eqref{eq:cfmvemdof3} are uni-solvent for the local virtual element space $V_k^m(K)$. The main difficulty is to count the dimension of $V_k^m(K)$.
To this end, we introduce data spaces
$$
\mathcal D(\partial K):=\prod_{\delta\in\mathcal F^n(K)}\prod_{j=0}^{m-1}\mathbb S_n(j) \times \prod_{r=1}^{n-1}\prod_{F\in\mathcal F^r(K)}\prod_{\alpha\in A_r\atop|\alpha|\leq m-1}\mathbb P_{k-2m+|\alpha|}(F),
$$
$$
\mathcal D(K):=\mathcal D(\partial K)\times\mathbb P_{k-2m}(K), \quad \widetilde{\mathcal D}(K):=\mathcal D(\partial K)\times\mathbb P_{k}(K).
$$
Clearly we have $\dim\mathcal D(K)=\#\mathcal N_k^m(K)$.
For simplicity, let notation $(d_{n}^{\delta, j}, d_{r}^{F, \alpha})\in\mathcal D(\partial K)$ mean
\begin{itemize}
\item $d_{n}^{\delta, j}\in\mathbb S_n(j)$ for each $\delta\in\mathcal F^n(K)$ and $j=0,1,\cdots,m-1$;
\item $d_{r}^{F, \alpha}\in\mathbb P_{k-2m+|\alpha|}(F)$ for each $F\in\mathcal F^r(K)$ with $r=1,\cdots, n-1$, $\alpha\in A_r$, and $|\alpha|\leq m-1$.
\end{itemize}
Notation $(d_{n}^{\delta, j}, d_{r}^{F, \alpha}, d_0)\in\mathcal D(K)$ means $(d_{n}^{\delta, j}, d_{r}^{F, \alpha})\in\mathcal D(\partial K)$ and $d_0\in\mathbb P_{k-2m}(K)$, and notation $(d_{n}^{\delta, j}, d_{r}^{F, \alpha}, d_0)\in\widetilde{\mathcal D}(K)$ is understood similarly.
We will show that both the mapping $\mathcal D_K: V_k^m(K)\to\mathcal D(K)$ given by 
$$
\mathcal D_Kv:=\bigg(\nabla^{j}v(\delta),\; Q_{k-2m+|\alpha|}^F\frac{\partial^{|\alpha|}v}{\partial \nu_F^{\alpha}},\; (-\Delta)^mv\bigg) \in \mathcal D(K) \quad\textrm{with}\quad v\in V_k^m(K),
$$
and the mapping $\widetilde{\mathcal D}_K: \widetilde{V}_k^m(K)\to\widetilde{\mathcal D}(K)$ given by 
$$
\widetilde{\mathcal D}_Kv:=\bigg(\nabla^{j}v(\delta),\; Q_{k-2m+|\alpha|}^F\frac{\partial^{|\alpha|}v}{\partial \nu_F^{\alpha}},\; (-\Delta)^mv\bigg) \in \widetilde{\mathcal D}(K) \quad\textrm{with}\quad v\in \widetilde{V}_k^m(K),
$$
are bijective. The idea of introducing data spaces can be found in \cite{ChenHuang2022}, and similar idea, i.e. degrees of freedom tuple, is advanced in \cite{AntoniettiManziniScacchiVerani2021}.

For a function $v\in H^m(K)$, the trace $\textrm{Tr}\,v:=\Big(v|_{\partial K}, \frac{\partial v}{\partial \nu_F}\Big|_{\partial K}, \cdots, \frac{\partial^{m-1}v}{\partial \nu_F^{m-1}}\Big|_{\partial K}\Big) \in H^{m-1/2}(\mathcal F^1(K))\times\cdots\times H^{1/2}(\mathcal F^1(K))$, where
$$
H^{s}(\mathcal F^1(K)):=\{v\in L^2(\partial K): v|_F\in H^{s}(F)\quad \forall~ F\in\mathcal F^1(K)\} \textrm{ for } s>0. 
$$
The trace space $\textrm{Tr}\,H^m(K)\neq H^{m-1/2}(\mathcal F^1(K))\times\cdots\times H^{1/2}(\mathcal F^1(K))$, since there exist some compatibility conditions among the components of $\textrm{Tr}\,v$ \cite{LambertiProvenzano2020}.
To present the characterization of the trace space $\textrm{Tr}\,H^m(K)$ in \cite{Agranovich2007,Agranovich2008,Verchota1990}, we first define the space of Whitney arrays 
\begin{align*}
\textrm{WA}(\partial K):=&\,\big\{\{g_{\alpha}\}_{\alpha\in A_n, |\alpha|\leq m-1}: g_{\alpha}\in H^1(\partial K)\;\; \forall~\alpha\in A_n \textrm{ with } |\alpha|\leq m-2, \\
&\qquad\qquad\qquad\qquad\quad\;\, g_{\alpha}\in H^{1/2}(\partial K)\;\; \forall~\alpha\in A_n \textrm{ with } |\alpha|= m-1, \\
&\quad\quad\quad\;\;\,\textrm{ and the compatibility conditions \eqref{eq:tracecompatiblecond} for } g_{\alpha} \textrm{ are satisfied}\big\},
\end{align*}
where the compatibility conditions are
\begin{equation}\label{eq:tracecompatiblecond}
(\nu_F)_j\partial_ig_{\alpha}-(\nu_F)_i\partial_jg_{\alpha}=(\nu_F)_jg_{\alpha+e_i}-(\nu_F)_ig_{\alpha+e_j} \;\;\textrm{ on each }F\in\mathcal F^1(K)
\end{equation}
for each $\alpha\in A_n, |\alpha|\leq m-2$ and $1\leq i\neq j\leq n$, $(\nu_F)_i=e_i\cdot\nu_F$ and $\partial_ig_{\alpha}=e_i\cdot\nabla g_{\alpha}$.

For $v\in H^m(K)$, clearly we have the array $\{\partial^{\alpha}v|_{\partial K}\}_{\alpha\in A_n, |\alpha|\leq m-1}\in \textrm{WA}(\partial K)$, where $\partial^{\alpha}v:=e^{\alpha}:\nabla^{|\alpha|}v$ with $e^{\alpha}:=e_1^{\alpha_1}\otimes\cdots\otimes e_n^{\alpha_n}$. And in this case, expressions in both sides of \eqref{eq:tracecompatiblecond} are two representations of some tangential derivative of the trace of $\partial^{\alpha}v$. Moreover, such a trace mapping is onto, which is listed in the following lemma.

\begin{lemma}[Theorem~5 in \cite{Agranovich2007}, Theorem~4 in \cite{Agranovich2008} and Theorem~R(m) in \cite{Verchota1990}]
\label{lem:inversetracethm}
Let $K\in\mathbb R^n$ be a polytope.
For each Whitney array $\{g_{\alpha}\}_{\alpha\in A_n, |\alpha|\leq m-1}\in \textrm{WA}(\partial K)$, there exists a function $v\in H^m(K)$ such that
$$
\partial^{\alpha}v|_{\partial K}=g_{\alpha}\quad\forall~\alpha\in A_n, |\alpha|\leq m-1.
$$
Moreover, there exists a linear and bounded operator from $\textrm{WA}(\partial K)$ to $H^m(K)$.
\end{lemma}

In the next two lemmas, we will construct a Whitney array for each data in $\mathcal D(\partial K)$.
\begin{lemma}\label{lem:dataglue}
Given data $(d_{n}^{\delta, j}, d_{r}^{F, \alpha})\in\mathcal D(\partial K)$, 
there exist $g_{n}^{\delta,j}\in\mathbb S_n(j)$ and $g_{r}^{F,j}\in H^{m-j}(F;\mathbb S_n(j))$ for any  $\delta\in\mathcal F^n(K)$, $F\in\mathcal F^r(K)$, $r=1,\cdots, n-1$, and $j=0,1,\cdots,m-1$ such that
\begin{enumerate}[(i)]
\item $g_{n}^{\delta,j}=d_{n}^{\delta, j}$ for each $\delta\in\mathcal F^n(K)$ and $j=0,1,\cdots,m-1$;
\item $g_{r}^{F,|\alpha|}:\nu_F^{\alpha}\in V_{k-|\alpha|}^{m-|\alpha|}(F)$ and 
$Q_{k-2m+|\alpha|}^F(g_{r}^{F,|\alpha|}:\nu_F^{\alpha})=d_{r}^{F,\alpha}$ for each $F\in\mathcal F^r(K)$, $\alpha\in A_r$, $|\alpha|\leq m-1$, $r=1,\cdots, n-1$;
\item for each $F\in\mathcal F^r(K)$, $r=1,\cdots, n-1$ and $j=0,1,\cdots,m-1$,
\begin{equation}\label{eq:20210425}
g_{r}^{F,j} = \sum_{\alpha\in A_{r}, \beta\in A_{n-r}\atop |\alpha|+|\beta|=j}\frac{j!}{\alpha!\beta!}\sym(\nu_F^{\alpha}\otimes t_{F}^{\beta})\frac{\partial^{|\beta|}}{\partial t_{F}^{\beta}}\big(g_{r}^{F,|\alpha|}:\nu_F^{\alpha}\big);   
\end{equation}
\item $g_{r}^{F,j}|_e=g_{r+s}^{e,j}$ for each $F\in\mathcal F^r(K)$, $e\in\mathcal F^s(F)$, $r=1,\cdots, n-1$, $s=1,\cdots, n-r$, and $j=0,1,\cdots,m-1$.
\end{enumerate}
\end{lemma}

\begin{remark}\rm
We will see in the proof of Lemma~\ref{lem:tidleVkdim} that $g_{r}^{F, j}=\nabla^jv^b|_F$ for some $v^b\in H^m(K)$.
In the colusions of Lemma~\ref{lem:dataglue}, $g_{r}^{F,|\alpha|}:\nu_F^{\alpha}\in V_{k-|\alpha|}^{m-|\alpha|}(F)$ is motivated by $(\nabla^{|\alpha|}v^b)|_F:\nu_F^{\alpha}=\frac{\partial^{|\alpha|}v^b}{\partial \nu_F^{\alpha}}\Big|_F\in V_{k-|\alpha|}^{m-|\alpha|}(F)$ in the definition of $\widetilde{V}_k^m(K)$, $Q_{k-2m+|\alpha|}^F(g_{r}^{F,|\alpha|}:\nu_F^{\alpha})=d_{r}^{F,\alpha}$ motivated by the DoFs \eqref{eq:cfmvemdof2}, and equation \eqref{eq:20210425} motivated by \eqref{eq:20210413}. It follows from \eqref{eq:20210425} that
$$
g_{r}^{F,j}:\sym(\nu_F^{\alpha}\otimes t_{F}^{\beta}) = \frac{\partial^{|\beta|}}{\partial t_{F}^{\beta}}\big(g_{r}^{F,|\alpha|}:\nu_F^{\alpha}\big).
$$
\end{remark}

\noindent\textit{Proof of Lemma~\ref{lem:dataglue}}.
First take $g_{n}^{\delta,j}=d_{n}^{\delta, j}$ for each $\delta\in\mathcal F^n(K)$ and $j=0,1,\cdots,m-1$.
For each $e\in\mathcal F^{n-1}(K)$, $\alpha\in A_{n-1}$ and $|\alpha|\leq m-1$, take $v_{n-1}^{e,\alpha}\in V_{k-|\alpha|}^{m-|\alpha|}(e)$ satisfying
\begin{equation*}
\left\{
\begin{aligned}
&\bigg(\frac{\partial^{j}v_{n-1}^{e,\alpha}}{\partial t_e^{j}}\bigg)(\delta)=g_{n}^{\delta, j+|\alpha|}:\sym(\nu_e^{\alpha}\otimes t_e^j)\quad \forall~\delta\in\mathcal F^1(e), j=0,1,\cdots,m-|\alpha|-1, \\ 
&Q_{k-2m+|\alpha|}^ev_{n-1}^{e,\alpha}=d_{n-1}^{e,\alpha}\quad\text{in
}e. 
\end{aligned}
\right.
\end{equation*}
For $j=0,1,\cdots, m-1$, inspired by \eqref{eq:20210413}, let
$$
g_{n-1}^{e,j} = \sum_{\alpha\in A_{n-1}, |\alpha|\leq j}\frac{j!}{\alpha!(j-|\alpha|)!}\sym(\nu_e^{\alpha}\otimes t_e^{j-|\alpha|})\bigg(\frac{\partial^{j-|\alpha|}v_{n-1}^{e,\alpha}}{\partial t_e^{j-|\alpha|}}\bigg).
$$ 
Then we have for any $\delta\in\mathcal F^1(e)$ that
\begin{align*}
g_{n-1}^{e,j}(\delta)=&\sum_{\alpha\in A_{n-1}, |\alpha|\leq j}\frac{j!}{\alpha!(j-|\alpha|)!}\sym(\nu_e^{\alpha}\otimes t_e^{j-|\alpha|})\big(g_{n}^{\delta, j}:\sym(\nu_e^{\alpha}\otimes t_e^{j-|\alpha|})\big) \\
=&\,g_{n}^{\delta, j}.
\end{align*}
And it follows
\begin{align*}
g_{n-1}^{e,|\alpha|}:\nu_e^{\alpha}=& \sum_{\beta\in A_{n-1}, |\beta|\leq |\alpha|}\frac{|\alpha|!}{\beta!(|\alpha|-|\beta|)!}\sym(\nu_e^{\beta}\otimes t_e^{|\alpha|-|\beta|})\bigg(\frac{\partial^{|\alpha|-|\beta|}v_{n-1}^{e,\beta}}{\partial t_e^{|\alpha|-|\beta|}}\bigg):\nu_e^{\alpha}\\
=&\, v_{n-1}^{e,\alpha}\in V_{k-|\alpha|}^{m-|\alpha|}(e)\qquad\qquad\forall~\alpha\in A_{n-1}, |\alpha|\leq m-1.   
\end{align*}
In turn we have
$$
Q_{k-2m+|\alpha|}^e(g_{n-1}^{e,|\alpha|}:\nu_e^{\alpha})=d_{n-1}^{e,\alpha},
$$
$$
g_{n-1}^{e,j} = \sum_{\alpha\in A_{n-1}, |\alpha|\leq j}\frac{j!}{\alpha!(j-|\alpha|)!}\sym(\nu_e^{\alpha}\otimes t_e^{j-|\alpha|})\frac{\partial^{j-|\alpha|}}{\partial t_e^{j-|\alpha|}}\big(g_{n-1}^{e,|\alpha|}:\nu_e^{\alpha}\big).
$$ 

Assume we have found $g_{s}^{e,j}\in H^{m-j}(e;\mathbb S_n(j))$ for any $e\in\mathcal F^s(K)$, $s=r+1,\cdots, n-1$, and $j=0,1,\cdots,m-1$ for $1\leq r\leq n-2$ satisfying
\begin{itemize}
\item $g_{s}^{e,|\alpha|}:\nu_e^{\alpha}\in V_{k-|\alpha|}^{m-|\alpha|}(e)$ and 
$Q_{k-2m+|\alpha|}^e(g_{s}^{e,|\alpha|}:\nu_e^{\alpha})=d_{s}^{e,\alpha}$ for each $e\in\mathcal F^s(K)$, $\alpha\in A_s$, $|\alpha|\leq m-1$, $s=r+1,\cdots, n-1$;
\item for each $e\in\mathcal F^s(K)$, $s=r+1,\cdots, n-1$ and $j=0,1,\cdots,m-1$,
$$
g_{s}^{e,j} = \sum_{\alpha\in A_{s}, \beta\in A_{n-s}\atop |\alpha|+|\beta|=j}\frac{j!}{\alpha!\beta!}\sym(\nu_e^{\alpha}\otimes t_{e}^{\beta})\frac{\partial^{|\beta|}}{\partial t_{e}^{\beta}}\big(g_{s}^{e,|\alpha|}:\nu_e^{\alpha}\big);
$$\item $g_{s}^{e,j}|_{e'}=g_{s+s'}^{e',j}$ for each $e\in\mathcal F^s(K)$, $e'\in\mathcal F^{s'}(e)$, $s=r+1,\cdots, n-1$, $s'=1,\cdots, n-s$.
\end{itemize}
Now consider the construction of $g_{r}^{F,j}$ for each $F\in\mathcal F^{r}(K)$ and $j=1,\cdots,m-1$.
To this end, for any $\alpha\in A_{r}$ and $|\alpha|\leq m-1$, by \eqref{eq:cfmvemfacesdof1}-\eqref{eq:cfmvemfacesdof3} let $v_{r}^{F,\alpha}\in V_{k-|\alpha|}^{m-|\alpha|}(F)$ be determined by
\begin{equation*}
\left\{
\begin{aligned}
&\bigg(\frac{\partial^{|\beta|}v_{r}^{F,\alpha}}{\partial\nu_{F,\delta}^{\beta}}\bigg)(\delta)=g_{n}^{\delta, |\beta|+|\alpha|}:\sym(\nu_F^{\alpha}\otimes\nu_{F,\delta}^{\beta})\quad \forall~\delta\in\mathcal F^{n-r}(F), \beta\in A_{n-r}, \\
&\qquad\qquad\qquad\qquad\qquad\qquad\qquad\qquad\qquad\qquad\qquad|\beta|\leq m-|\alpha|-1, \\
&Q_{k-2m+|\alpha|+|\beta|}^e\bigg(\frac{\partial^{|\beta|}v_{r}^{F,\alpha}}{\partial\nu_{F,e}^{\beta}}-g_{r+s}^{e, |\beta|+|\alpha|}:\sym(\nu_F^{\alpha}\otimes\nu_{F,e}^{\beta})\bigg) = 0\quad \forall~e\in\mathcal F^{s}(F),  \\
&\qquad\qquad\qquad\qquad\qquad\qquad s=1,\cdots,n-r-1,\beta\in A_{s},  |\beta|\leq m-|\alpha|-1, \\
&Q_{k-2m+|\alpha|}^Fv_{r}^{F,\alpha}=d_{r}^{F,\alpha}\quad\text{in
}F. 
\end{aligned}
\right.
\end{equation*}
Noting that $\frac{\partial^{|\beta|}v_{r}^{F,\alpha}}{\partial\nu_{F,e}^{\beta}}, g_{r+s}^{e, |\beta|+|\alpha|}:\sym(\nu_F^{\alpha}\otimes\nu_{F,e}^{\beta}) \in V_{k-|\alpha|-|\beta|}^{m-|\alpha|-|\beta|}(e)$ for each $e\in\mathcal F^{s}(F)$, and they share the same values of the DoFs, it follows 
\begin{equation}\label{eq:20210415}
\frac{\partial^{|\beta|}v_{r}^{F,\alpha}}{\partial\nu_{F,e}^{\beta}}=g_{r+s}^{e, |\beta|+|\alpha|}:\sym(\nu_F^{\alpha}\otimes\nu_{F,e}^{\beta}) \quad\forall~e\in\mathcal F^{s}(F).
\end{equation}
For $j=0,1,\cdots, m-1$, let
\begin{equation}\label{eq:20210601}
g_{r}^{F,j} = \sum_{\alpha\in A_{r}, \beta\in A_{n-r}\atop |\alpha|+|\beta|=j}\frac{j!}{\alpha!\beta!}\sym(\nu_F^{\alpha}\otimes t_{F}^{\beta})\bigg(\frac{\partial^{|\beta|}v_{r}^{F,\alpha}}{\partial t_{F}^{\beta}}\bigg).
\end{equation}
Then
$$
g_{r}^{F,|\alpha|}:\nu_F^{\alpha}=v_{r}^{F,\alpha}\in V_{k-|\alpha|}^{m-|\alpha|}(F),
$$
which yields
$$
Q_{k-2m+|\alpha|}^F(g_{r}^{F,|\alpha|}:\nu_F^{\alpha})=d_{r}^{F,\alpha},
$$
$$
g_{r}^{F,j} = \sum_{\alpha\in A_{r}, \beta\in A_{n-r}\atop |\alpha|+|\beta|=j}\frac{j!}{\alpha!\beta!}\sym(\nu_F^{\alpha}\otimes t_{F}^{\beta})\frac{\partial^{|\beta|}}{\partial t_{F}^{\beta}}\big(g_{r}^{F,|\alpha|}:\nu_F^{\alpha}\big).
$$
For each $e\in\mathcal F^{s}(F)$ with $s=1,\cdots,n-r$, it follows from \eqref{eq:20210415} that
\begin{align*}
g_{r}^{F,j}|_e& = \sum_{\alpha\in A_{r}, \beta\in A_{s}, \gamma\in A_{n-r-s}\atop |\alpha|+|\beta|+|\gamma|=j}\frac{j!}{\alpha!\beta!\gamma!}\sym(\nu_F^{\alpha}\otimes \nu_{F,e}^{\beta}\otimes t_{e}^{\gamma})\bigg(\frac{\partial^{|\beta|+|\gamma|}v_{r}^{F,\alpha}}{\partial t_{e}^{\gamma}\partial\nu_{F,e}^{\beta}}\bigg)\bigg|_e \\
& = \sum_{\alpha\in A_{r}, \beta\in A_{s}, \gamma\in A_{n-r-s}\atop |\alpha|+|\beta|+|\gamma|=j}\frac{j!}{\alpha!\beta!\gamma!}\sym(\nu_F^{\alpha}\otimes \nu_{F,e}^{\beta}\otimes t_{e}^{\gamma})\frac{\partial^{|\gamma|}}{\partial t_{e}^{\gamma}}\big(g_{r+s}^{e, |\alpha|+|\beta|}:(\nu_F^{\alpha}\otimes\nu_{F,e}^{\beta})\big) \\
& = \sum_{\alpha\in A_{r+s}, \gamma\in A_{n-r-s}\atop |\alpha|+|\gamma|=j}\frac{j!}{\alpha!\gamma!}\sym(\nu_e^{\alpha}\otimes t_{e}^{\gamma})\frac{\partial^{|\gamma|}}{\partial t_{e}^{\gamma}}\big(g_{r+s}^{e, |\alpha|}:\nu_e^{\alpha}\big) \\
&=g_{r+s}^{e,j}.
\end{align*}
Finally we finish the proof by the mathematical induction.
$\hfill\Box$

\begin{lemma}\label{lem:datawhitneyarray}
Given data $(d_{n}^{\delta, j}, d_{r}^{F, \alpha})\in\mathcal D(\partial K)$, 
let $g_{1}^{F,j}\in H^{m-j}(F;\mathbb S_n(j))$ for any $F\in\mathcal F^1(K)$ and $j=0,1,\cdots,m-1$ be defined in Lemma~\ref{lem:dataglue}. For each $\alpha\in A_n, |\alpha|\leq m-1$, define $g_{\alpha}\in L^2(\partial K)$ by 
$g_{\alpha}|_F=g_{1}^{F,|\alpha|}:e^{\alpha}$ for each $F\in\mathcal F^1(K)$.
Then $\{g_{\alpha}\}_{\alpha\in A_n, |\alpha|\leq m-1}\in \textrm{WA}(\partial K)$.
\end{lemma}
\begin{proof}
By (ii) and (iv) in Lemma~\ref{lem:dataglue}, we have $g_{\alpha}|_F\in H^{m-|\alpha|}(F)$ and $g_{\alpha}\in H^1(\partial K)$ for each $F\in\mathcal F^1(K)$, and $\alpha\in A_n, |\alpha|\leq m-1$.

Next we check the compatibility conditions in \eqref{eq:tracecompatiblecond}. Noting that 
$$
\partial_ig_{\alpha}=(\nu_F)_i\frac{\partial g_{\alpha}}{\partial \nu_F} +\sum_{\ell=1}^{n-1}(t_{F,\ell})_i\frac{\partial g_{\alpha}}{\partial t_{F,\ell}},\quad \partial_jg_{\alpha}=(\nu_F)_j\frac{\partial g_{\alpha}}{\partial \nu_F} +\sum_{\ell=1}^{n-1}(t_{F,\ell})_j\frac{\partial g_{\alpha}}{\partial t_{F,\ell}},
$$
we get 
\begin{equation}\label{eq:20210602}
(\nu_F)_j\partial_ig_{\alpha}-(\nu_F)_i\partial_jg_{\alpha}=\sum_{\ell=1}^{n-1}\big((\nu_F)_j(t_{F,\ell})_i-(\nu_F)_i(t_{F,\ell})_j\big)\frac{\partial g_{\alpha}}{\partial t_{F,\ell}}.  
\end{equation}
On the other side, it follows from \eqref{eq:20210601} that
$$
g_{1}^{F,|\alpha|+1} = \sum_{\beta\in A_{n-1}\atop |\beta|\leq |\alpha|+1}\frac{(|\alpha|+1)!}{(|\alpha|+1-|\beta|)!\beta!}\sym(\nu_F^{|\alpha|+1-|\beta|}\otimes t_{F}^{\beta})\frac{\partial^{|\beta|}v_r^{F,|\alpha|+1-|\beta|}}{\partial t_{F}^{\beta}},
$$
$$
g_{1}^{F,|\alpha|} = \sum_{\beta\in A_{n-1}\atop |\beta|\leq |\alpha|}\frac{|\alpha|!}{(|\alpha|-|\beta|)!\beta!}\sym(\nu_F^{|\alpha|-|\beta|}\otimes t_{F}^{\beta})\frac{\partial^{|\beta|}v_r^{F,|\alpha|-|\beta|}}{\partial t_{F}^{\beta}}.
$$
Hence for $\ell=1,\cdots, n-1$, we get
$$
\resizebox{\textwidth}{!}{$
\begin{aligned}
&\quad\; g_{1}^{F,|\alpha|+1}:(e^{\alpha}\otimes t_{F,\ell}) \\
&=\sum_{\beta\in A_{n-1}\atop |\beta|\leq |\alpha|+1}\frac{(|\alpha|+1)!}{(|\alpha|+1-|\beta|)!\beta!}\sym(\nu_F^{|\alpha|+1-|\beta|}\otimes t_{F}^{\beta}):(e^{\alpha}\otimes t_{F,\ell})\frac{\partial^{|\beta|}v_r^{F,|\alpha|+1-|\beta|}}{\partial t_{F}^{\beta}} \\
&=\sum_{\beta\in A_{n-1}\atop |\beta|\leq |\alpha|}\frac{(|\alpha|+1)!}{(|\alpha|-|\beta|)!\beta!(\beta_{\ell}+1)}\sym(\nu_F^{|\alpha|-|\beta|}\otimes t_{F}^{\beta}\otimes t_{F,\ell}):(e^{\alpha}\otimes t_{F,\ell})\frac{\partial^{|\beta|+1}v_r^{F,|\alpha|-|\beta|}}{\partial t_{F,\ell}\partial t_{F}^{\beta}} \\
&=\sum_{\beta\in A_{n-1}\atop |\beta|\leq |\alpha|}\frac{|\alpha|!}{(|\alpha|-|\beta|)!\beta!}\sym(\nu_F^{|\alpha|-|\beta|}\otimes t_{F}^{\beta}):e^{\alpha}\frac{\partial^{|\beta|+1}v_r^{F,|\alpha|-|\beta|}}{\partial t_{F,\ell}\partial t_{F}^{\beta}} \\
&= \frac{\partial(g_{1}^{F,|\alpha|}:e^{\alpha})}{\partial t_{F,\ell}}= \frac{\partial g_{\alpha}}{\partial t_{F,\ell}}.
\end{aligned}
$}
$$
By $e_i=(\nu_F)_i\nu_F+\sum\limits_{\ell=1}^{n-1}(t_{F,\ell})_it_{F,\ell}$, we get
\begin{align*}
g_{\alpha+e_i}&=g_{1}^{F,|\alpha|+1}:(e^{\alpha}\otimes e_i) \\
&=(\nu_F)_ig_{1}^{F,|\alpha|+1}:(e^{\alpha}\otimes\nu_F) + \sum_{\ell=1}^{n-1}(t_{F,\ell})_ig_{1}^{F,|\alpha|+1}:(e^{\alpha}\otimes t_{F,\ell}).
\end{align*}
Thus it holds
\begin{align*}
(\nu_F)_jg_{\alpha+e_i}-(\nu_F)_ig_{\alpha+e_j}&=\sum_{\ell=1}^{n-1}\big((\nu_F)_j(t_{F,\ell})_i-(\nu_F)_i(t_{F,\ell})_j\big)g_{1}^{F,|\alpha|+1}:(e^{\alpha}\otimes t_{F,\ell}) \\
&=\sum_{\ell=1}^{n-1}\big((\nu_F)_j(t_{F,\ell})_i-(\nu_F)_i(t_{F,\ell})_j\big)\frac{\partial g_{\alpha}}{\partial t_{F,\ell}}.
\end{align*}
Therefore we conlude the compatibility conditions in \eqref{eq:tracecompatiblecond} from \eqref{eq:20210602}.
\end{proof}

\subsection{Uni-solvence of virtual elements}\label{subsec:vemunisolvence}
With previous preparations, we will prove the uni-solvence of the $H^m$-conforming virtual elements $(K, \mathcal N_k^m(K), V_k^m(K))$ in arbitrary dimension in this subsection.
\begin{lemma}\label{lem:tidleVkdim}
The mapping $\widetilde{\mathcal D}_K: \widetilde{V}_k^m(K)\to\widetilde{\mathcal D}(K)$ is onto. Consequently
$$
\dim \widetilde{V}_k^m(K)\geq\dim \widetilde{\mathcal D}(K). 
$$
\end{lemma}
\begin{proof}
Take any data $(d_{n}^{\delta, j}, d_{r}^{F, \alpha}, d_0)\in\widetilde{\mathcal D}(K)$.
Due to $(d_{n}^{\delta, j}, d_{r}^{F, \alpha})\in\mathcal D(\partial K)$,
let $\{g_{\alpha}\}_{\alpha\in A_n, |\alpha|\leq m-1}\in \textrm{WA}(\partial K)$
be defined in Lemma~\ref{lem:datawhitneyarray}.
By Lemma~\ref{lem:inversetracethm},
there exists $v^b\in H^m(K)$ such that
$$
\partial^{\alpha}v^b|_{\partial K}=g_{\alpha}\quad\forall~\alpha\in A_n, |\alpha|\leq m-1.
$$
Then 
$$
\partial^{\alpha}v^b|_F=g_{1}^{F,|\alpha|}:e^{\alpha}\quad\forall~F\in\mathcal F^1(K), \alpha\in A_n, |\alpha|\leq m-1,
$$
which implies
$$
\nabla^jv^b|_F=g_{1}^{F, j}\quad\forall~F\in\mathcal F^1(K), j=1,\cdots, m-1.
$$
And we get from (iv) in Lemma~\ref{lem:dataglue} that
$$
\nabla^jv^b|_F=g_{r}^{F, j}\quad\forall~F\in\mathcal F^r(K), r=1,\cdots, n,\; j=1,\cdots, m-1.
$$
On the other side,
there exists unique $v^0\in H_0^m(K)$ determined by 
$$
(-\Delta)^mv^0=d_0-(-\Delta)^mv^b.
$$
Take $v=v^0+v^b\in H^m(K)$.
We have $(-\Delta)^mv=d_0\in\mathbb P_{k}(K)$, and
\begin{equation}\label{eq:20210527-1}  
\nabla^jv|_F=\nabla^jv^b|_F=g_{r}^{F, j}\quad\forall~F\in\mathcal F^r(K), r=1,\cdots, n, \; j=1,\cdots, m-1.
\end{equation}
It follows from the last identity and (iv) in Lemma~\ref{lem:dataglue} that
$(\nabla^jv)|_{\mathcal S_K^r}\in H^1(\mathcal S_K^r;\mathbb S_{n}(j))$ for $r=1,\cdots, n-1$, $j=0,\cdots,m-1$.
And thanks to (ii) in Lemma~\ref{lem:dataglue},
\begin{equation}\label{eq:20210527-2}  
\frac{\partial^{|\alpha|}v}{\partial \nu_F^{\alpha}}\Big|_F=\frac{\partial^{|\alpha|}v^b}{\partial \nu_F^{\alpha}}\Big|_F=g_{r}^{F, |\alpha|}:\nu_F^{\alpha}\in V_{k-|\alpha|}^{m-|\alpha|}(F)
\end{equation}
for any $F\in\mathcal F^{r}(K)$, $r=1,\cdots, n-1$, $\alpha\in A_r$,  and  $|\alpha|\leq m-1$.
Thus
$v\in \widetilde{V}_k^m(K)$. And it follows from \eqref{eq:20210527-1}-\eqref{eq:20210527-2}, (i) and (ii) in Lemma~\ref{lem:dataglue} that $\widetilde{\mathcal D}_Kv=(d_{n}^{\delta, j}, d_{r}^{F, \alpha}, d_0)$. 
\end{proof}

\begin{lemma}\label{lem:tildevemunisolvence}
The following DoFs $\widetilde{\mathcal N}_k^m(K)$
\begin{align*}
h_K^j\nabla^{j}v(\delta) & \quad\forall~\delta\in\mathcal F^{n}(K), \;j=0,1,\cdots,m-1, \\
\frac{h_K^{|\alpha|}}{|F|}(\frac{\partial^{|\alpha|}v}{\partial \nu_F^{\alpha}}, q)_F & \quad\forall~q\in\mathbb P_{k-2m+|\alpha|}(F), F\in\mathcal F^{r}(K), \;r=1,\cdots,n-1, \\
&\quad \quad \alpha\in A_r, \textrm{ and } |\alpha|\leq m-1, \\
\frac{1}{|K|}(v, q)_K & \quad\forall~q\in\mathbb P_{k}(K) 
\end{align*}
are uni-solvent for the local virtual element space $\widetilde{V}_k^m(K)$.
Consequently the mapping $\widetilde{\mathcal D}_K: \widetilde{V}_k^m(K)\to\widetilde{\mathcal D}(K)$ is bijective.
\end{lemma}
\begin{proof}
Due to Lemma~\ref{lem:tidleVkdim}, we have $\dim\widetilde{V}_k^m(K)\geq\#\widetilde{\mathcal N}_k^m(K)$. Assume $v\in\widetilde{V}_k^m(K)$ and all the DoFs in $\widetilde{\mathcal N}_k^m(K)$ vanish. By the recursive definition of $\widetilde{V}_k^m(K)$, it follows from the vanishing DoFs on the boundary in $\widetilde{\mathcal N}_k^m(K)$ that $v\in H_0^m(K)$. Employing the integration by parts, we get from $(-\Delta)^mv|_K\in\mathbb P_k(K)$ that 
$$
\|\nabla^mv\|_{0,K}^2=(\nabla^mv,\nabla^mv)_K=(v, (-\Delta)^mv)_K=0.
$$
Thus $v=0$.
\end{proof}

\begin{lemma}
The DoFs~\eqref{eq:cfmvemdof1}-\eqref{eq:cfmvemdof3}, i.e. $\mathcal N_k^m(K)$, are uni-solvent for the local virtual element space $V_k^m(K)$. Consequently the mapping $\mathcal D_K: V_k^m(K)\to\mathcal D(K)$ is bijective.
\end{lemma}
\begin{proof}
By the definition of $\dim V_k^m(K)$, it holds $\dim V_k^m(K)\geq \#\mathcal N_k^m(K)$.
Assume $v\in V_k^m(K)$ and the DoFs~\eqref{eq:cfmvemdof1}-\eqref{eq:cfmvemdof3} vanish.
Notice that $\Pi_k^Kv=0$. Hence
$$
(v, q)_K=0\quad\forall~q\in\mathbb P_k(K),
$$
which together with Lemma~\ref{lem:tildevemunisolvence} yields $v=0$.
\end{proof}

As the two dimensioncal case, it holds
\begin{equation}\label{eq:QkKrepresent}
Q_k^Kv= \Pi_k^Kv + Q_{k-2m}^Kv-Q_{k-2m}^K\Pi_k^Kv\quad\forall~v\in V_k^m(K).   
\end{equation}
Hence $Q_k^Kv$ is computable using only the DoFs~\eqref{eq:cfmvemdof1}-\eqref{eq:cfmvemdof3} for any $v\in V_k^m(K)$.
And then $Q_{k}^K(\nabla^jv)$ is computable using only the DoFs~\eqref{eq:cfmvemdof1}-\eqref{eq:cfmvemdof3} for any $v\in V_k^m(K)$ and $j=1, \cdots, m$. As a result of \eqref{eq:QkKrepresent}, we have
\begin{equation}\label{eq:QkKrepresentkerpi}
Q_k^Kv= Q_{k-2m}^Kv\quad\forall~v\in \ker(\Pi_k^K)\cap V_k^m(K),   
\end{equation}
where $\ker(\Pi_k^K)\cap V_{k}^{m}(K):=\{v\in V_{k}^{m}(K): \Pi_k^Kv=0\}$.

The $H^2$-conforming virtual elements in three dimensions have been constructed in \cite{BeiraodaVeigaDassiRusso2020}.
\begin{remark}\rm
For the lowest degree case $k=m$, the DoFs \eqref{eq:cfmvemdof2}-\eqref{eq:cfmvemdof3} disappear, and $\mathcal N_k^m(K)$ will reduce to
$$
h_K^j\nabla^{j}v(\delta) \quad\forall~\delta\in\mathcal F^{n}(K), \;j=0,1,\cdots,m-1.
$$
\end{remark}

\begin{remark}\rm
When $k=m$ and $K\subset\mathbb R^n$ is a simplex, 
$$
\dim V_k^m(K)=(n+1)\dim\mathbb P_{m-1}(K).
$$
As a comparison, the dimension of the lowest degree $H^m$-conforming finite element in \cite{HuLinWu2021} is $\dim \mathbb P_{2^n(m-1)+1}(K)$, which is much larger than $\dim V_k^m(K)$. 
\end{remark}

\section{Inverse inequality and norm equivalences}\label{sec:invnorm}

We will derive the inverse inequality and several norm equivalences of the virtual elements $(K, \mathcal N_k^m(K), V_k^m(K))$ by the mathematical induction in this section, which are vitally important in the error analysis of virtual element methods.
Henceforth,  we always assume mesh conditions (A1) and (A2) hold, and polytope $K\in\mathcal T_h$.

\subsection{Inverse inequality}

We first employ the multiplicative trace inequality, the inverse trace theorem and the inverse inequality for polynomials to prove the inverse inequality for $V_k^m(K)$ through the mathematical induction.
\begin{lemma}\label{lem:inverseinequality0}
Let $F\in\mathcal F^r(K)$ with $r=0,1,\cdots, n-1$, $\ell=2,\cdots, m$ and integer $k_{\ell}\geq\ell$.
Let $v\in V_{k_{\ell}}^{\ell}(F)$ and positive integer $j\leq\ell-1$.
If $|v|_{j+1,F}\lesssim h_F^{-1}|v|_{j,F}$, then
\begin{equation*}
|v|_{j,F}\lesssim h_F^{-1}|v|_{j-1,F}.
\end{equation*}
\end{lemma}
\begin{proof}
It follows from the integration by parts,
the multiplicative trace inequality~\eqref{L2tracemulti} and the assumption $|v|_{j+1,F}\lesssim h_F^{-1}|v|_{j,F}$ that
\begin{align*}
|v|_{j,F}^2&=(\nabla_F^jv, \nabla_F^jv)_F=-(\Delta_F\nabla_F^{j-1}v, \nabla_F^{j-1}v)_F+\sum_{e\in\mathcal F^1(F)}(\frac{\partial}{\partial\nu_{F,e}}\nabla_F^{j-1}v, \nabla_F^{j-1}v)_{e} \\
&\lesssim |v|_{j+1,F}|v|_{j-1,F} +\|\nabla_F^{j}v\|_{0,\partial F}\|\nabla_F^{j-1}v\|_{0,\partial F} \\
&\lesssim |v|_{j+1,F}|v|_{j-1,F} \\
&\quad\quad+h_F^{-1}|v|_{j, F}^{1/2}|v|_{j-1, F}^{1/2}(|v|_{j, F}^{1/2}+h_F^{1/2}|v|_{j+1, F}^{1/2})(|v|_{j-1, F}^{1/2}+h_F^{1/2}|v|_{j, F}^{1/2}) \\
&\lesssim h_F^{-1}|v|_{j,F}|v|_{j-1,F}+h_F^{-1}|v|_{j, F}|v|_{j-1, F}^{1/2}(|v|_{j-1, F}^{1/2}+h_F^{1/2}|v|_{j, F}^{1/2})\\
&\lesssim h_F^{-1}|v|_{j,F}|v|_{j-1,F}+h_F^{-1/2}|v|_{j, F}^{3/2}|v|_{j-1, F}^{1/2}.
\end{align*}
Thus we end the proof by applying the Young's inequality to the last inequality.
\end{proof}


\begin{lemma}\label{lem:20210419}
Let $F\in\mathcal F^r(K)$ with $r=0,1,\cdots, n-2$.
Assume on each $e\in \mathcal F^1(F)$, it holds for any positive integer $\ell\leq m$ and integer $k_{\ell}\geq\ell$ that
\begin{equation}\label{eq:20210419}
|w|_{j,e}\lesssim h_e^{i-j}|w|_{i,e}\quad \forall~w\in V_{k_{\ell}}^{\ell}(e),\; 0\leq i< j\leq \ell.
\end{equation}
Then we have for any positive integer $\ell\leq m$ and integer $k_{\ell}\geq\ell$ that
\begin{equation}\label{eq:20210419-1}
|v|_{j,F}\lesssim h_F^{i-j}|v|_{i,F}\quad \forall~v\in V_{k_{\ell}}^{\ell}(F),\; 0\leq i< j\leq \ell.
\end{equation}
\end{lemma}
\begin{proof}
Let $v^b\in H^{\ell}(F)$ be the solution of the polyharmonic equation with nonhomogeneous Dirichlet boundary condition
\begin{equation*}
\left\{
\begin{aligned}
&(-\Delta_F)^{\ell}v^b=0\quad\,\text{ in
}F,  \\
& \frac{\partial^{j}v^b}{\partial \nu_{F,e}^{j}}=\frac{\partial^{j}v}{\partial \nu_{F,e}^{j}}  \quad\text{ on each }e\in\mathcal F^1(F), j=0,1,\cdots, \ell-1.
\end{aligned}
\right.
\end{equation*}
By \eqref{eq:20210413} we have $\nabla_F^j v^b|_{\partial F}=\nabla_F^j v|_{\partial F}$ for $j=0,1,\cdots, \ell-1$.
It is easy to check that
$$
|v^b|_{\ell,F}=\inf_{\phi\in H^{\ell}(F)\atop \nabla_F^j \phi|_{\partial F}=\nabla_F^j v|_{\partial F}, j=0,1,\cdots, \ell-1} |\phi|_{\ell, F}.
$$
On the other side,
due to Lemma~\ref{lem:inversetracethm}
and the Lipschitz isomorphism \cite{BrennerSung2018}, 
there exists $\phi\in H^{\ell}(F)$ such that $\frac{\partial^{j}\phi}{\partial \nu_{F,e}^{j}}=\frac{\partial^{j}v}{\partial \nu_{F,e}^{j}}$ ($j=0,1,\cdots, \ell-1$) for each $e\in\mathcal F^1(F)$, and
$$
|\phi|_{\ell,F}\lesssim \|\nabla_F^{\ell-1}\phi\|_{1/2,\partial F}.
$$ 
Hence $|v^b|_{\ell,F}\lesssim \|\nabla_F^{\ell-1}v\|_{1/2,\partial F}$.
By the space interpolation theory \cite{AdamsFournier2003,BrennerScott2008},
$$
|v^b|_{\ell,F}\lesssim \|\nabla_F^{\ell-1}v\|_{0,\partial F}^{1/2}(h_F^{-1/2}\|\nabla_F^{\ell-1}v\|_{0,\partial F}^{1/2}+|\nabla_F^{\ell-1}v|_{1,\partial F}^{1/2}).
$$ 
By \eqref{eq:20210413} and \eqref{eq:20210419} with $w=\frac{\partial^{j}v}{\partial \nu_{F,e}^{j}}\in V_{k_{\ell}-j}^{\ell-j}(e)$, 
\begin{align*}
|\nabla_F^{\ell-1}v|_{1,\partial F}^2&=\sum_{e\in\mathcal F^1(F)}\left|\sum_{\alpha\in A_{n-r-1},\, j\in\mathbb N \atop |\alpha|+j= \ell-1}\frac{(\ell-1)!}{j!\alpha!}\sym(\nu_{F,e}^{j}\otimes t_{e}^{\alpha})\frac{\partial^{\ell-1}v}{\partial t_{e}^{\alpha}\partial \nu_{F,e}^{j}}\right|_{1,e}^2 \\
&\lesssim \sum_{e\in\mathcal F^1(F)}\sum_{\alpha\in A_{n-r-1},\, j\in\mathbb N \atop |\alpha|+j= \ell-1}\bigg|\frac{\partial^{\ell-1}v}{\partial t_{e}^{\alpha}\partial \nu_{F,e}^{j}}\bigg|_{1,e}^2 \\
&\lesssim h_F^{-2}\sum_{e\in\mathcal F^1(F)}\sum_{\alpha\in A_{n-r-1},\, j\in\mathbb N \atop |\alpha|+j= \ell-1}\bigg\|\frac{\partial^{\ell-1}v}{\partial t_{e}^{\alpha}\partial \nu_{F,e}^{j}}\bigg\|_{0,e}^2 \lesssim h_F^{-2}\|\nabla_F^{\ell-1}v\|_{0,\partial F}^2.
\end{align*}
Thus 
\begin{equation}\label{eq:20210527-3}
|v^b|_{\ell,F}\lesssim h_F^{-1/2}\|\nabla_F^{\ell-1}v\|_{0,\partial F}.  
\end{equation}
Applying the multiplicative trace inequality~\eqref{L2tracemulti}, we get
\begin{align}
|v^b|_{\ell,F}&\lesssim h_F^{-1}|v|_{\ell-1, F}^{1/2}(|v|_{\ell-1,F}^{1/2}+h_F^{1/2}|v|_{\ell,F}^{1/2}) \notag\\
&\lesssim h_F^{-1}|v|_{\ell-1,F} + h_F^{-1/2}|v|_{\ell-1,F}^{1/2}|v|_{\ell,F}^{1/2}. \label{eq:20210418-1}
\end{align}
Notice that $v-v^b\in H_0^{\ell}(F)$. It follows from the inverse inequality for polynomials~\eqref{eq:polyinverse}, the fact $(-\Delta_F)^{\ell}v^b=0$ and the Poincar\'e-Friedrichs inequality \eqref{eq:Poincare-Friedrichs2} that
\begin{align}
|v-v^b|_{\ell,F}^2&=(\nabla_F^{\ell}(v-v^b), \nabla_F^{\ell}(v-v^b))_F=((-\Delta_F)^{\ell}(v-v^b), v-v^b)_F \label{eq:20210421-1}\\
&\leq \|(-\Delta_F)^{\ell}v\|_{0,F}\|v-v^b\|_{0,F} \notag\\
&\lesssim h_F^{-\ell-1}\|(-\Delta_F)^{\ell}v\|_{-\ell-1,F}\|v-v^b\|_{0,F} \notag\\
&\lesssim h_F^{-\ell-1}|v|_{\ell-1,F}\|v-v^b\|_{0,F} \notag\\
&\lesssim h_F^{-1}|v|_{\ell-1,F}|v-v^b|_{\ell,F},  \notag
\end{align}
which means 
$|v-v^b|_{\ell,F}\lesssim h_F^{-1}|v|_{\ell-1,F}$.
Together with \eqref{eq:20210418-1}, we obtain
$$
|v|_{\ell,F}\leq |v^b|_{\ell,F}+|v-v^b|_{\ell,F}\lesssim h_F^{-1}|v|_{\ell-1,F} + h_F^{-1/2}|v|_{\ell-1,F}^{1/2}|v|_{\ell,F}^{1/2}.
$$
Thus 
$$
|v|_{\ell,F}\lesssim h_F^{-1}|v|_{\ell-1,F}.
$$
Finally we achieve \eqref{eq:20210419-1} from Lemma~\ref{lem:inverseinequality0}.
\end{proof}

\begin{lemma}[Inverse inequality]
For each $F\in\mathcal F^r(K)$ with $r=0,1,\cdots, n-1$, it holds the inverse inequality
\begin{equation}\label{eq:inverseinequality}
|v|_{j,F}\lesssim h_F^{i-j}|v|_{i,F}\quad \forall~v\in V_k^m(K),\; 0\leq i< j\leq m.
\end{equation}
\end{lemma}
\begin{proof}
On each $e\in \mathcal F^{n-1}(K)$, since $V_{k_{\ell}}^{\ell}(e)=\mathbb P_{\max\{k_{\ell}, 2\ell-1\}}(e)$ for any positive integer $\ell\leq m$ and integer $k_{\ell}\geq\ell$, applying the inverse inequality for polynomials~\eqref{eq:polyinverse}, 
\begin{equation*}
|w|_{j,e}\lesssim h_e^{i-j}|w|_{i,e}\quad \forall~w\in V_{k_{\ell}}^{\ell}(e),\; 0\leq i< j\leq \ell.
\end{equation*}
Thus \eqref{eq:inverseinequality} follows from Lemma~\ref{lem:20210419} by applying the mathematical induction.
\end{proof}

\subsection{Norm equivalences}
Next we show several norm equivalences on the finite dimensional spaces $V_{k}^{m}(K)$, $\ker(Q_k^K)\cap V_{k}^{m}(K)$ and $\ker(\Pi_k^K)\cap V_{k}^{m}(K)$, where $\ker(Q_k^K)\cap V_{k}^{m}(K):=\{v\in V_{k}^{m}(K): Q_k^Kv=0\}$.

\begin{lemma}
Let $F\in\mathcal F^r(K)$ with $r=0,1,\cdots, n-1$.
It holds for any positive integer $\ell\leq m$, integer $k_{\ell}\geq\ell$ and non-negative integer $j\leq \ell-1$ that
\begin{equation}\label{eq:normequivalencevertices}
|v|_{j,F}\eqsim h_F|v|_{j+1,F} + h_F^{(n-r)/2}\Bigg|\sum_{\delta\in\mathcal F^{n-r}(F)}(\nabla_F^{j}v)(\delta)\Bigg| \quad\forall~v\in V_{k_{\ell}}^{\ell}(F).
\end{equation}
\end{lemma}
\begin{proof}
Thanks to the trace inequality \eqref{L2trace} and the inverse inequality \eqref{eq:inverseinequality}, it is sufficient to prove
\begin{equation}\label{eq:20210426-1}
|v|_{j,F}\lesssim h_F|v|_{j+1,F} + h_F^{(n-r)/2}\Bigg|\sum_{\delta\in\mathcal F^{n-r}(F)}(\nabla_F^{j}v)(\delta)\Bigg| \quad\forall~v\in V_{k_{\ell}}^{\ell}(F).
\end{equation}
Noting that $\nabla_F^j(Q_j^Fv)$ is constant,
\begin{align*}
|Q_j^Fv|_{j,F}&=\|\nabla_F^j(Q_j^Fv)\|_{0,F}\eqsim  h_F^{(n-r)/2}|\nabla_F^j(Q_j^Fv)| \\
&\leq h_F^{(n-r)/2}\Bigg|\nabla_F^j(Q_j^Fv) - \frac{1}{\#\mathcal F^{n-r}(F)}\sum_{\delta\in\mathcal F^{n-r}(F)}(\nabla_F^{j}v)(\delta)\Bigg| \\
&\quad+\frac{1}{\#\mathcal F^{n-r}(F)}h_F^{(n-r)/2}\Bigg|\sum_{\delta\in\mathcal F^{n-r}(F)}(\nabla_F^{j}v)(\delta)\Bigg|  \\
&\lesssim h_F^{(n-r)/2}\sum_{\delta\in\mathcal F^{n-r}(F)}|\nabla_F^j(Q_j^Fv) - (\nabla_F^{j}v)(\delta)|  \\
&\quad+h_F^{(n-r)/2}\Bigg|\sum_{\delta\in\mathcal F^{n-r}(F)}(\nabla_F^{j}v)(\delta)\Bigg|.     
\end{align*}
Applying the trace inequality \eqref{L2trace} and the inverse inequality \eqref{eq:inverseinequality} recursively,
\begin{align*}
&\quad h_F^{(n-r)/2}\sum_{\delta\in\mathcal F^{n-r}(F)}|\nabla_F^j(Q_j^Fv-v)(\delta)| \\
&\lesssim h_F^{(n-r-1)/2}\sum_{e\in\mathcal F^{n-r-1}(F)}(\|\nabla_F^j(Q_j^Fv-v)\|_{0,e}+h_e|\nabla_F^j(Q_j^Fv-v)|_{1,e}) \\
&\lesssim h_F^{(n-r-1)/2}\sum_{e\in\mathcal F^{n-r-1}(F)}\|\nabla_F^j(Q_j^Fv-v)\|_{0,e} \\
&\lesssim \cdots\lesssim h_F^{1/2}\sum_{e\in\mathcal F^{1}(F)}\|\nabla_F^j(Q_j^Fv-v)\|_{0,e}\lesssim |Q_j^Fv-v|_{j,F}. 
\end{align*}
Then we get from the last two inequalities that
$$
|v|_{j,F}\leq |v-Q_j^Fv|_{j,F}+|Q_j^Fv|_{j,F}\lesssim |v-Q_j^Fv|_{j,F}+h_F^{(n-r)/2}\Bigg|\sum_{\delta\in\mathcal F^{n-r}(F)}(\nabla_F^{j}v)(\delta)\Bigg|,
$$
which together with \eqref{eq:l2projectionestimate} gives \eqref{eq:20210426-1}.
\end{proof}

\begin{lemma}
Let $F\in\mathcal F^r(K)$ with $r=0,1,\cdots, n-2$.
Assume on each $e\in \mathcal F^1(F)$, it holds for any positive integer $\ell\leq m$ and integer $k_{\ell}\geq\ell$ that
\begin{align}
\|w\|_{0,e}^2&\lesssim \|Q_{k_{\ell}-2\ell}^ew\|_{0,e}^2 + \sum_{\delta\in\mathcal F^{n-r-1}(e)}\sum_{i=0}^{\ell-1}h_e^{n-r-1+2i}|\nabla_e^{i}w(\delta)|^2 \notag\\
&\quad
+ \sum_{s=1}^{n-r-2}\sum_{e'\in\mathcal F^{s}(e)}\sum_{\alpha\in A_{s}, |\alpha|\leq \ell-1}h_e^{s+2|\alpha|}\bigg\|Q_{k_{\ell}-2\ell+|\alpha|}^{e'}\frac{\partial^{|\alpha|}w}{\partial \nu_{e,e'}^{\alpha}}\bigg\|_{0,e'}^2\;\forall~w\in V_{k_{\ell}}^{\ell}(e). \label{eq:20210420}
\end{align}
Then we have for any positive integer $\ell\leq m$, integer $k_{\ell}\geq\ell$, non-negative integer $j\leq \ell$ and $v\in V_{k_{\ell}}^{\ell}(F)$ that
\begin{align}
h_F^{2j}|\Pi_{k_{\ell}}^Fv|_{j,F}^2&\lesssim \|Q_{k_{\ell}-2\ell}^Fv\|_{0,F}^2 + \sum_{\delta\in\mathcal F^{n-r}(F)}\sum_{i=0}^{\ell-1}h_F^{n-r+2i}|\nabla_F^{i}v(\delta)|^2 \notag\\
&\quad
+ \sum_{s=1}^{n-r-1}\sum_{e'\in\mathcal F^{s}(F)}\sum_{\alpha\in A_{s}, |\alpha|\leq \ell-1}h_F^{s+2|\alpha|}\bigg\|Q_{k_{\ell}-2\ell+|\alpha|}^{e'}\frac{\partial^{|\alpha|}v}{\partial \nu_{F,e'}^{\alpha}}\bigg\|_{0,e'}^2. \label{eq:20210421-4}
\end{align}
\end{lemma}
\begin{proof}
Due to \eqref{eq:normequivalencevertices}, it is sufficient to prove \eqref{eq:20210421-4} with $j=\ell$.
It follows from
\eqref{eq:greenidentity} and the inverse inequality for polynomials~\eqref{eq:polyinverse} that
\begin{align*}
|\Pi_{k_{\ell}}^Fv|_{\ell,F}^2&=(\nabla_F^{\ell}v, \nabla_F^{\ell}\Pi_{k_{\ell}}^Fv)_F \\
&=(v, (-\Delta_F)^{\ell}\Pi_{k_{\ell}}^Fv)_F+\sum_{i=0}^{\ell-1}\sum_{e\in\mathcal F^1(F)}(\nabla_F^{i}v, \nabla_F^{i}(-\Delta_F)^{\ell-i-1}\frac{\partial\Pi_{k_{\ell}}^Fv}{\partial\nu_{F,e}})_{e}  \\
&=(Q_{k_{\ell}-2\ell}^Fv, (-\Delta_F)^{\ell}\Pi_{k_{\ell}}^Fv)_F+\sum_{i=0}^{\ell-1}\sum_{e\in\mathcal F^1(F)}(\nabla_F^{i}v, \nabla_F^{i}(-\Delta_F)^{\ell-i-1}\frac{\partial\Pi_{k_{\ell}}^Fv}{\partial\nu_{F,e}})_{e} \\
&\lesssim \|Q_{k_{\ell}-2\ell}^Fv\|_{0,F}|\Pi_{k_{\ell}}^Fv|_{2\ell,F}+\sum_{i=0}^{\ell-1}\sum_{e\in\mathcal F^1(F)}\|\nabla_F^{i}v\|_{0,e}\|\nabla_F^{2\ell-i-1}\Pi_{k_{\ell}}^Fv\|_{0,e} \\
&\lesssim h_F^{-\ell}\|Q_{k_{\ell}-2\ell}^Fv\|_{0,F}|\Pi_{k_{\ell}}^Fv|_{\ell,F}+\sum_{i=0}^{\ell-1}\sum_{e\in\mathcal F^1(F)}h_F^{-\ell+i+1/2}\|\nabla_F^{i}v\|_{0,e}|\Pi_{k_{\ell}}^Fv|_{\ell,F}.
\end{align*}
Hence we have
\begin{equation}\label{eq:20210421-5}
h_F^{2\ell}|\Pi_{k_{\ell}}^Fv|_{\ell,F}^2\lesssim \|Q_{k_{\ell}-2\ell}^Fv\|_{0,F}^2 +\sum_{i=0}^{\ell-1}\sum_{e\in\mathcal F^1(F)}h_F^{2i+1}\|\nabla_F^{i}v\|_{0,e}^2.
\end{equation}
For $i=0,\cdots,\ell-1$ and each $e\in\mathcal F^1(F)$, it follows from \eqref{eq:20210413} and the inverse inequality \eqref{eq:inverseinequality} that
\begin{align*}
h_F^{2i+1}\|\nabla_F^{i}v\|_{0,e}^2&=h_F^{2i+1}\Bigg\|\sum_{\alpha\in A_{n-r-1},|\alpha|+j= i}\frac{i!}{j!\alpha!}\sym(\nu_{F,e}^{j}\otimes t_e^{\alpha})\frac{\partial^{i}v}{\partial t_e^{\alpha}\partial \nu_{F,e}^{j}}\Bigg\|_{0,e}^2 \\
&\lesssim \sum_{\alpha\in A_{n-r-1},|\alpha|+j= i}h_e^{2i+1}\bigg\|\frac{\partial^{i}v}{\partial t_e^{\alpha}\partial \nu_{F,e}^{j}}\bigg\|_{0,e}^2 \lesssim \sum_{j=0}^ih_e^{2j+1}\bigg\|\frac{\partial^{j}v}{\partial \nu_{F,e}^{j}}\bigg\|_{0,e}^2. 
\end{align*}
Noting that $\frac{\partial^{j}v}{\partial \nu_{F,e}^{j}}|_e\in V_{k_{\ell}-j}^{\ell-j}(e)$, we get from \eqref{eq:20210420} with $w=\frac{\partial^{j}v}{\partial \nu_{F,e}^{j}}|_e$ that
\begin{align}
&h_F^{2i+1}\|\nabla_F^{i}v\|_{0,e}^2\lesssim\sum_{j=0}^ih_e^{2j+1}\bigg\|Q_{k_{\ell}-2\ell+j}^e\frac{\partial^{j}v}{\partial \nu_{F,e}^{j}}\bigg\|_{0,e}^2 \notag\\
&\quad+\sum_{j=0}^{i}\sum_{s=1}^{n-r-1}\sum_{e'\in\mathcal F^{s}(e)}\sum_{\alpha\in A_{s}\atop |\alpha|\leq \ell-j-1}h_e^{s+1+2(|\alpha|+j)}\bigg\|Q_{k_{\ell}-2\ell+j+|\alpha|}^{e'}\frac{\partial^{|\alpha|+j}v}{\partial \nu_{e,e'}^{\alpha}\partial \nu_{F,e}^{j}}\bigg\|_{0,e'}^2. \label{eq:20210421-6}
\end{align}
Changing $e'\in\mathcal F^{s}(e)$ with $s=0,\cdots, n-r-1$ to $e'\in\mathcal F^{s}(F)$ with $s=1,\cdots, n-r$, we have
\begin{equation*}
h_F^{2i+1}\|\nabla_F^{i}v\|_{0,e}^2\lesssim  \sum_{j=0}^{i}\sum_{s=1}^{n-r}\sum_{e'\in\mathcal F^{s}(F)}\sum_{\alpha\in A_{s}, |\alpha|\leq \ell-1}h_F^{s+2|\alpha|}\bigg\|Q_{k_{\ell}-2\ell+|\alpha|}^{e'}\frac{\partial^{|\alpha|}v}{\partial \nu_{F,e'}^{\alpha}}\bigg\|_{0,e'}^2. 
\end{equation*}
Thus \eqref{eq:20210421-4} with $j=\ell$ follows from \eqref{eq:20210421-5}.
\end{proof}

\begin{lemma}\label{lem:20210420}
Let $F\in\mathcal F^r(K)$ with $r=0,1,\cdots, n-2$.
Assume \eqref{eq:20210420} holds on each $e\in \mathcal F^1(F)$ for any positive integer $\ell\leq m$, integer $k_{\ell}\geq\ell$.
Then we have for any positive integer $\ell\leq m$, integer $k_{\ell}\geq\ell$, $j\leq \ell$ and $v\in V_{k_{\ell}}^{\ell}(F)$ that
\begin{align}
h_F^{2j}|v|_{j,F}^2&\lesssim \|Q_{k_{\ell}-2\ell}^Fv\|_{0,F}^2 + \sum_{\delta\in\mathcal F^{n-r}(F)}\sum_{i=0}^{\ell-1}h_F^{n-r+2i}|\nabla_F^{i}v(\delta)|^2 \notag\\
&\quad
+ \sum_{s=1}^{n-r-1}\sum_{e'\in\mathcal F^{s}(F)}\sum_{\alpha\in A_{s}, |\alpha|\leq \ell-1}h_F^{s+2|\alpha|}\bigg\|Q_{k_{\ell}-2\ell+|\alpha|}^{e'}\frac{\partial^{|\alpha|}v}{\partial \nu_{F,e'}^{\alpha}}\bigg\|_{0,e'}^2\;\,\forall~v\in V_{k_{\ell}}^{\ell}(F). \label{eq:20210420-1}
\end{align}
\end{lemma}
\begin{proof}
Let $v^b\in H^{\ell}(F)$ be defined as in Lemma~\ref{lem:20210419}.
By \eqref{eq:20210527-3} and \eqref{eq:20210413}, we have
$$
|v^b|_{\ell,F}\lesssim h_F^{-1/2}\|\nabla_F^{\ell-1}v\|_{0,\partial F}\lesssim
h_F^{-1/2}\sum_{e\in\mathcal F^1(F)}\sum_{e\in\mathcal F^1(F)}\sum_{\alpha\in A_{n-r-1},\, j\in\mathbb N \atop |\alpha|+j= \ell-1}\bigg\|\frac{\partial^{\ell-1}v}{\partial t_{e}^{\alpha}\partial \nu_{F,e}^{j}}\bigg\|_{0,e},
$$
which together with the inverse inequality \eqref{eq:inverseinequality} yields
$$
h_F^{2\ell}|v^b|_{\ell,F}^2\lesssim \sum_{e\in\mathcal F^1(F)}\sum_{j=0}^{\ell-1}h_e^{2j+1}\bigg\|\frac{\partial^{j}v}{\partial \nu_{F,e}^{j}}\bigg\|_{0,e}^2.
$$
Similarly as \eqref{eq:20210421-6},
by $\frac{\partial^{j}v}{\partial \nu_{F,e}^{j}}|_e\in V_{k_{\ell}-j}^{\ell-j}(e)$, we get from \eqref{eq:20210420} with $w=\frac{\partial^{j}v}{\partial \nu_{F,e}^{j}}|_e$ that
\begin{align*}
&h_F^{2\ell}|v^b|_{\ell,F}^2
\lesssim \sum_{e\in\mathcal F^1(F)}\sum_{j=0}^{\ell-1}h_e^{2j+1}\bigg\|Q_{k_{\ell}-2\ell+j}^e\frac{\partial^{j}v}{\partial \nu_{F,e}^{j}}\bigg\|_{0,e}^2 \\
&+\sum_{e\in\mathcal F^1(F)}\sum_{j=0}^{\ell-1}\sum_{s=1}^{n-r-1}\sum_{e'\in\mathcal F^{s}(e)}\sum_{\alpha\in A_{s}\atop |\alpha|\leq \ell-j-1}h_e^{s+1+2(|\alpha|+j)}\bigg\|Q_{k_{\ell}-2\ell+j+|\alpha|}^{e'}\frac{\partial^{|\alpha|+j}v}{\partial \nu_{e,e'}^{\alpha}\partial \nu_{F,e}^{j}}\bigg\|_{0,e'}^2.
\end{align*}
Hence
\begin{align}
h_F^{2\ell}|v^b|_{\ell,F}^2&\lesssim  \sum_{s=1}^{n-r}\sum_{e'\in\mathcal F^{s}(F)}\sum_{\alpha\in A_{s}, |\alpha|\leq \ell-1}h_F^{s+2|\alpha|}\bigg\|Q_{k_{\ell}-2\ell+|\alpha|}^{e'}\frac{\partial^{|\alpha|}v}{\partial \nu_{F,e'}^{\alpha}}\bigg\|_{0,e'}^2. \label{eq:20210420-2}
\end{align}

Notice that $\nabla_F^i v^b|_{\partial F}=\nabla_F^i v|_{\partial F}$ for $i=0,1,\cdots, \ell-1$.
For $j=1,\cdots, \ell-1$,
\begin{align*}
|v^b|_{j,F}^2&=(\nabla_F^jv^b, \nabla_F^jv^b)_F=(\nabla_F^jv^b, \nabla_F(\nabla_F^{j-1}v^b-Q_0^F\nabla_F^{j-1}v^b))_F \\
&=-(\Delta_F\nabla_F^{j-1}v^b, \nabla_F^{j-1}v^b-Q_0^F\nabla_F^{j-1}v^b)_F \\
&\quad+\sum_{e\in\mathcal F^1(F)}(\frac{\partial}{\partial\nu_{F,e}}\nabla_F^{j-1}v, \nabla_F^{j-1}v^b-Q_0^F\nabla_F^{j-1}v^b)_{e} \\
&\lesssim |v^b|_{j+1,F}\|\nabla_F^{j-1}v^b-Q_0^F\nabla_F^{j-1}v^b\|_{0,F} \\
&\quad+\sum_{e\in\mathcal F^1(F)}\|\nabla_F^{j}v\|_{0,e}\|\nabla_F^{j-1}v^b-Q_0^F\nabla_F^{j-1}v^b\|_{0,e}.
\end{align*}
It follows from the last inequality and \eqref{eq:l2projectionestimate} that
$$
h_F^{j}|v^b|_{j,F}\lesssim h_F^{j+1}|v^b|_{j+1,F}+h_F^{j+1/2}\|\nabla_F^{j}v\|_{0,\partial F}\quad \textrm{ for } j=1,\cdots, \ell-1.
$$
Hence we achieve from the last inequality, \eqref{eq:20210420} and \eqref{eq:20210420-2} that 
\begin{align}
h_F^{2j}|v^b|_{j,F}^2&\lesssim \sum_{s=1}^{n-r}\sum_{e'\in\mathcal F^{s}(F)}\sum_{\alpha\in A_{s}, |\alpha|\leq \ell-1}h_F^{s+2|\alpha|}\bigg\|Q_{k_{\ell}-2\ell+|\alpha|}^{e'}\frac{\partial^{|\alpha|}v}{\partial \nu_{F,e'}^{\alpha}}\bigg\|_{0,e'}^2  \label{eq:20210420-3}
\end{align}
for $j=1,\cdots, \ell$.
Take some $\delta\in\mathcal F^{n-r}(F)$.
Applying the trace inequality \eqref{L2trace} recursively, \eqref{eq:l2projectionestimate} and the inverse inequality \eqref{eq:inverseinequality},
\begin{align*}
|(v^b-Q_0^Fv^b)(\delta)|^2&\lesssim \sum_{e\in\mathcal F^{n-r-1}(F)}h_e^{-1}\|v^b-Q_0^Fv^b\|_{0,e}^2+\sum_{e\in\mathcal F^{n-r-1}(F)}h_e|v^b|_{1,e}^2 \\
&\lesssim\cdots\lesssim h_F^{r-n}\|v^b-Q_0^Fv^b\|_{0,F}^2+\sum_{s=0}^{n-r-1}\sum_{e\in\mathcal F^{s}(F)}h_e^{r-n+s+2}|v^b|_{1,e}^2 \\
&\lesssim h_F^{r-n+2}|v^b|_{1,F}^2+ \sum_{s=1}^{n-r-1}\sum_{e\in\mathcal F^{s}(F)}h_e^{r-n+s}\|v\|_{0,e}^2.
\end{align*}
This implies
\begin{align*}
|Q_0^Fv^b|^2&=|(Q_0^Fv^b)(\delta)|^2\lesssim |v^b(\delta)|^2+|(v^b-Q_0^Fv^b)(\delta)|^2 \\
&\lesssim |v(\delta)|^2+h_F^{r-n+2}|v^b|_{1,F}^2+ \sum_{s=1}^{n-r-1}\sum_{e\in\mathcal F^{s}(F)}h_e^{r-n+s}\|v\|_{0,e}^2,
\end{align*}
which together with \eqref{eq:l2projectionestimate}, the trace inequality \eqref{L2trace} and the inverse inequality \eqref{eq:inverseinequality} means
\begin{align*}
\|v^b\|_{0,F}^2&\lesssim \|v^b-Q_0^Fv^b\|_{0,F}^2+\|Q_0^Fv^b\|_{0,F}^2\lesssim h_F^2|v^b|_{1,F}^2 + h_F^{n-r}|Q_0^Fv^b|^2  \\
&\lesssim h_F^{n-r}|v(\delta)|^2+h_F^{2}|v^b|_{1,F}^2+\sum_{s=1}^{n-r-1}\sum_{e\in\mathcal F^{s}(F)}h_e^{s}\|v\|_{0,e}^2 \\
&\lesssim h_F^{n-r}|v(\delta)|^2+h_F^{2}|v^b|_{1,F}^2+\sum_{e\in\mathcal F^{1}(F)}h_e\|v\|_{0,e}^2.
\end{align*}
Then we drive from \eqref{eq:20210420-3} and \eqref{eq:20210420} that 
\begin{align}
\|v^b\|_{0,F}^2&\lesssim \sum_{s=1}^{n-r}\sum_{e'\in\mathcal F^{s}(F)}\sum_{\alpha\in A_{s}, |\alpha|\leq \ell-1}h_F^{s+2|\alpha|}\bigg\|Q_{k_{\ell}-2\ell+|\alpha|}^{e'}\frac{\partial^{|\alpha|}v}{\partial \nu_{F,e'}^{\alpha}}\bigg\|_{0,e'}^2.  \label{eq:20210420-4}
\end{align}

On the other side,
it follows from \eqref{eq:20210421-1}, the inverse inequality for polynomials~\eqref{eq:polyinverse} and the fact $(-\Delta_F)^{\ell}v^b=0$ that
\begin{align*}
|v-v^b|_{\ell,F}^2&=((-\Delta_F)^{\ell}(v-v^b), v-v^b)_F=((-\Delta_F)^{\ell}(v-v^b), Q_{k_{\ell}}^F(v-v^b))_F \\
&\leq \|(-\Delta_F)^{\ell}(v-v^b)\|_{0,F}\|Q_{k_{\ell}}^F(v-v^b)\|_{0,F} \\
&\lesssim h_F^{-\ell}\|(-\Delta_F)^{\ell}(v-v^b)\|_{-\ell,F}(\|Q_{k_{\ell}}^Fv\|_{0,F} + \|v^b\|_{0,F}) \\
&\lesssim h_F^{-\ell}|v-v^b|_{\ell,F}(\|Q_{k_{\ell}}^Fv\|_{0,F} + \|v^b\|_{0,F}), 
\end{align*}
which yields
$$
h_F^{2\ell}|v-v^b|_{\ell,F}^2\lesssim \|Q_{k_{\ell}}^Fv\|_{0,F}^2 + \|v^b\|_{0,F}^2.
$$
Combined with the fact $v-v^b\in H_0^{\ell}(F)$ and the Poincar\'e-Friedrichs inequality~\eqref{eq:Poincare-Friedrichs2}, we get
$$
h_F^{2j}|v|_{j,F}^2\lesssim h_F^{2j}|v-v^b|_{j,F}^2+h_F^{2j}|v^b|_{j,F}^2\lesssim \|Q_{k_{\ell}}^Fv\|_{0,F}^2 + \|v^b\|_{0,F}^2+h_F^{2j}|v^b|_{j,F}^2
$$
for $j=0,\cdots, \ell$.
By \eqref{eq:QkKrepresent}, $Q_{k_{\ell}}^Fv=Q_{k_{\ell}-2\ell}^Fv+(I-Q_{k_{\ell}-2\ell}^F)\Pi_{k_{\ell}}^Fv$, hence
$$
h_F^{2j}|v|_{j,F}^2\lesssim \|Q_{k_{\ell}-2\ell}^Fv\|_{0,F}^2 + \|\Pi_{k_{\ell}}^Fv\|_{0,F}^2 + \|v^b\|_{0,F}^2+h_F^{2j}|v^b|_{j,F}^2.
$$
Finally we conclude \eqref{eq:20210420-1} from the last inequality, \eqref{eq:20210421-4}, \eqref{eq:20210420-4} and \eqref{eq:20210420-3}.
\end{proof}

\begin{lemma}[Norm equivalence of virtual element spaces]
For any $v\in V_{k}^{m}(K)$ and $j=0,\cdots, m$, we have 
\begin{align}
h_K^{2j}|v|_{j,K}^2&\lesssim \|Q_{k-2m}^Kv\|_{0,K}^2 + \sum_{\delta\in\mathcal F^{n}(K)}\sum_{i=0}^{m-1}h_K^{n+2i}|\nabla^{i}v(\delta)|^2 \notag\\
&\quad
+ \sum_{r=1}^{n-1}\sum_{F\in\mathcal F^{r}(K)}\sum_{\alpha\in A_{r}, |\alpha|\leq m-1}h_K^{r+2|\alpha|}\bigg\|Q_{k-2m+|\alpha|}^{F}\frac{\partial^{|\alpha|}v}{\partial \nu_{F}^{\alpha}}\bigg\|_{0,F}^2, \label{eq:normequivalence1}
\end{align}
and
\begin{align}
\|v\|_{0,K}^2&\eqsim \|Q_{k-2m}^Kv\|_{0,K}^2 + \sum_{\delta\in\mathcal F^{n}(K)}\sum_{i=0}^{m-1}h_K^{n+2i}|\nabla^{i}v(\delta)|^2 \notag\\
&\quad
+ \sum_{r=1}^{n-1}\sum_{F\in\mathcal F^{r}(K)}\sum_{\alpha\in A_{r}, |\alpha|\leq m-1}h_K^{r+2|\alpha|}\bigg\|Q_{k-2m+|\alpha|}^{F}\frac{\partial^{|\alpha|}v}{\partial \nu_{F}^{\alpha}}\bigg\|_{0,F}^2. \label{eq:normequivalence}
\end{align}
\end{lemma}
\begin{proof}
Clearly the inequality \eqref{eq:normequivalence1} holds for $n=1$ since $V_{k}^{m}(K)=\mathbb P_{\max\{k, 2m-1\}}(K)$.
Then we obtain \eqref{eq:normequivalence1} for general $n$ from 
Lemma~\ref{lem:20210420} and the mathematical induction.

For the norm equivalence \eqref{eq:normequivalence}, due to \eqref{eq:normequivalence1} with $j=0$, it is sufficient to prove
\begin{align}
&\|Q_{k-2m}^Kv\|_{0,K}^2 + \sum_{\delta\in\mathcal F^{n}(K)}\sum_{i=0}^{m-1}h_K^{n+2i}|\nabla^{i}v(\delta)|^2 \notag\\
&\quad
+ \sum_{r=1}^{n-1}\sum_{F\in\mathcal F^{r}(K)}\sum_{\alpha\in A_{r}, |\alpha|\leq m-1}h_K^{r+2|\alpha|}\bigg\|Q_{k-2m+|\alpha|}^{F}\frac{\partial^{|\alpha|}v}{\partial \nu_{F}^{\alpha}}\bigg\|_{0,F}^2 \lesssim \|v\|_{0,K}^2. \label{eq:20210421-2}
\end{align}
Applying the trace inequality \eqref{L2trace} and the inverse inequality \eqref{eq:inverseinequality} recursively,
\begin{align*}
\sum_{\delta\in\mathcal F^{n}(K)}\sum_{i=0}^{m-1}h_K^{n+2i}|\nabla^{i}v(\delta)|^2&\lesssim \sum_{e\in\mathcal F^{n-1}(K)}\sum_{i=0}^{m-1}h_K^{n+2i}(h_e^{-1}\|\nabla^{i}v\|_{0,e}^2 + h_e |\nabla^{i}v|_{1,e}^2) \\
&\lesssim \sum_{e\in\mathcal F^{n-1}(K)}\sum_{i=0}^{m-1}h_K^{n-1+2i}\|\nabla^{i}v\|_{0,e}^2 \\
&\lesssim \cdots \lesssim \sum_{F\in\mathcal F^{1}(K)}\sum_{i=0}^{m-1}h_K^{1+2i}\|\nabla^{i}v\|_{0,F}^2 \\
& \lesssim \sum_{i=0}^{m-1}h_K^{2i}\|\nabla^{i}v\|_{0,K}^2 \lesssim \|v\|_{0,K}^2.
\end{align*}
Similarly we have
\begin{align*}
&\quad \sum_{r=1}^{n-1}\sum_{F\in\mathcal F^{r}(K)}\sum_{\alpha\in A_{r}, |\alpha|\leq m-1}h_K^{r+2|\alpha|}\bigg\|Q_{k-2m+|\alpha|}^{F}\frac{\partial^{|\alpha|}v}{\partial \nu_{F}^{\alpha}}\bigg\|_{0,F}^2 \\
&\leq \sum_{r=1}^{n-1}\sum_{F\in\mathcal F^{r}(K)}\sum_{i=0}^{m-1}h_K^{r+2i}\|\nabla^{i}v\|_{0,F}^2\lesssim \|v\|_{0,K}^2.
\end{align*}
Thus \eqref{eq:20210421-2} follows from the last two inequalities and $\|Q_{k-2m}^Kv\|_{0,K}\leq \|v\|_{0,K}$.
\end{proof}

Now we present the norm equivalences of the kernel space of the local $L^2$-projector $Q_k^K$ and
the local $H^m$-projector $\Pi_k^K$, which only involve the boundary DoFs.

\begin{lemma}[Norm equivalence of the kernel space of $Q_k^K$]
For any $v\in \ker(Q_k^K)\cap V_{k}^{m}(K)$ and $j=0,\cdots, m$, we have 
\begin{align}
h_K^{2j}|v|_{j,K}^2&\eqsim \sum_{r=1}^{n}\sum_{F\in\mathcal F^{r}(K)}\sum_{\alpha\in A_{r}, |\alpha|\leq m-1}h_K^{r+2|\alpha|}\bigg\|Q_{k-2m+|\alpha|}^{F}\frac{\partial^{|\alpha|}v}{\partial \nu_{F}^{\alpha}}\bigg\|_{0,F}^2. \label{eq:normequivalence-kerQk}
\end{align}
\end{lemma}
\begin{proof}
Noting that $Q_{k-2m}^Kv=0$, we achieve from \eqref{eq:normequivalence} that
\begin{equation*}
\|v\|_{0,K}^2\eqsim \sum_{r=1}^{n}\sum_{F\in\mathcal F^{r}(K)}\sum_{\alpha\in A_{r}, |\alpha|\leq m-1}h_K^{r+2|\alpha|}\bigg\|Q_{k-2m+|\alpha|}^{F}\frac{\partial^{|\alpha|}v}{\partial \nu_{F}^{\alpha}}\bigg\|_{0,F}^2. 
\end{equation*}
On the other side, we get from the inverse inequality \eqref{eq:inverseinequality} and \eqref{eq:l2projectionestimate} that
$$
h_K^j|v|_{j,K}\eqsim \|v\|_{0,K}\quad\forall~j=0,\cdots, m.
$$
Therefore \eqref{eq:normequivalence-kerQk} holds.
\end{proof}

Noting that $\sum\limits_{\delta\in\mathcal F^{n}(K)}(\nabla^{j}v)(\delta)=0$ for
any $v\in \ker(\Pi_k^K)\cap V_{k}^{m}(K)$ and $j=0,1,\cdots, m$,
it holds from \eqref{eq:normequivalencevertices} that 
\begin{equation}\label{eq:normequivalence-kerpi0}
h_K^{j}|v|_{j,K}\eqsim h_K^{j+1}|v|_{j+1,K} \quad\forall~v\in\ker(\Pi_k^K)\cap V_{k}^{m}(K), j=0,1,\cdots, m-1.
\end{equation}

\begin{lemma}[Norm equivalence of the kernel space of $\Pi_k^K$]
For any $v\in \ker(\Pi_k^K)\cap V_{k}^{m}(K)$ and $j=0,\cdots, m$, we have 
\begin{align}
h_K^{2j}|v|_{j,K}^2&\eqsim \sum_{r=1}^{n}\sum_{F\in\mathcal F^{r}(K)}\sum_{\alpha\in A_{r}, |\alpha|\leq m-1}h_K^{r+2|\alpha|}\bigg\|Q_{k-2m+|\alpha|}^{F}\frac{\partial^{|\alpha|}v}{\partial \nu_{F}^{\alpha}}\bigg\|_{0,F}^2. \label{eq:normequivalence-kerpi1}
\end{align}
\end{lemma}
\begin{proof}
Due to \eqref{eq:normequivalence} and \eqref{eq:normequivalence-kerpi0}, 
it is sufficient to prove
\begin{equation}\label{eq:20210422-1}
h_K^{2m}|v|_{m,K}^2\lesssim \sum_{r=1}^{n}\sum_{F\in\mathcal F^{r}(K)}\sum_{\alpha\in A_{r}, |\alpha|\leq m-1}h_K^{r+2|\alpha|}\bigg\|Q_{k-2m+|\alpha|}^{F}\frac{\partial^{|\alpha|}v}{\partial \nu_{F}^{\alpha}}\bigg\|_{0,F}^2. 
\end{equation}

Let $v^b\in H^{m}(K)$ be defined as in Lemma~\ref{lem:20210419} with $r=0$, $\ell=m$ and $k_{\ell}=k$.
By~\eqref{eq:20210420-3} and \eqref{eq:20210420-4}, 
we have
\begin{equation}\label{eq:20210422-2}
h_K^{2j}|v^b|_{j,K}^2\lesssim \sum_{r=1}^{n}\sum_{F\in\mathcal F^{r}(K)}\sum_{\alpha\in A_{r}, |\alpha|\leq m-1}h_K^{r+2|\alpha|}\bigg\|Q_{k-2m+|\alpha|}^{F}\frac{\partial^{|\alpha|}v}{\partial \nu_{F}^{\alpha}}\bigg\|_{0,F}^2   
\end{equation}
for $j=0,\cdots, m$.
On the other hand, according to Lemma 4.8 in \cite{Huang2020},
there exists $p\in\mathbb P_{k}(K)$ satisfying
\begin{equation}\label{eqn:20181113-2}
(-\Delta)^mp=Q_{k-2m}^K((-\Delta)^mv),
\end{equation}
\begin{equation}\label{eq:20190222-1}
|p|_{m,K}\lesssim h_K^{m}\|Q_{k-2m}^K((-\Delta)^mv)\|_{0,K}\leq h_K^{m}\|(-\Delta)^mv\|_{0,K}\lesssim |v|_{m,K}.
\end{equation}
Noting that $(\nabla^mp, \nabla^mv)_K=0$, $v-v^b\in H_0^m(K)$ and $(-\Delta)^mv^b=0$,
\begin{align*}
|v-v^b|_{m,K}^2&=(\nabla^m(v-v^b), \nabla^m(v-v^b))_K \\
&=(\nabla^m(v-v^b-p), \nabla^m(v-v^b))_K-(\nabla^mp, \nabla^mv^b)_K \\
&=((-\Delta)^m(v-v^b-p), v-v^b)_K-(\nabla^mp, \nabla^mv^b)_K  \\
&=((-\Delta)^m(v-p), v)_K-((-\Delta)^m(v-v^b-p), v^b)_K-(\nabla^mp, \nabla^mv^b)_K. 
\end{align*}
For the first term in the right hand side of the last equation,
it follows from \eqref{eq:QkKrepresentkerpi} and \eqref{eqn:20181113-2} that
\begin{align*}
((-\Delta)^m(v-p), v)_K&=((-\Delta)^m(v-p), Q_k^Kv)_K=((-\Delta)^m(v-p), Q_{k-2m}^Kv)_K \\
&=(Q_{k-2m}^K((-\Delta)^m(v-p)), v)_K=0.
\end{align*}
Then we acquire from the inverse inequality for polynomials~\eqref{eq:polyinverse} and \eqref{eq:20190222-1} that
\begin{align*}
|v-v^b|_{m,K}^2&\leq \|(-\Delta)^m(v-v^b-p)\|_{0,K}\|v^b\|_{0,K}+|p|_{m,K}|v^b|_{m,K} \\
&\lesssim h_K^{-m}(|v-v^b|_{m,K}+|p|_{m,K})\|v^b\|_{0,K}+|p|_{m,K}|v^b|_{m,K} \\
&\lesssim h_K^{-m}(|v-v^b|_{m,K}+|v|_{m,K})\|v^b\|_{0,K}+|v|_{m,K}|v^b|_{m,K} \\
&\lesssim h_K^{-m}|v-v^b|_{m,K}\|v^b\|_{0,K}+|v|_{m,K}(h_K^{-m}\|v^b\|_{0,K}+|v^b|_{m,K}),
\end{align*}
which gives
$$
|v-v^b|_{m,K}^2\lesssim h_K^{-2m}\|v^b\|_{0,K}^2+|v|_{m,K}(h_K^{-m}\|v^b\|_{0,K}+|v^b|_{m,K}).
$$
Hence
\begin{align*}
h_K^{2m}|v|_{m,K}^2&\lesssim h_K^{2m}|v-v^b|_{m,K}^2+h_K^{2m}|v^b|_{m,K}^2 \\
&\lesssim h_K^{m}|v|_{m,K}(\|v^b\|_{0,K}+h_K^{m}|v^b|_{m,K})+\|v^b\|_{0,K}^2+h_K^{2m}|v^b|_{m,K}^2,   
\end{align*}
which means
$$
h_K^{2m}|v|_{m,K}^2\lesssim\|v^b\|_{0,K}^2+h_K^{2m}|v^b|_{m,K}^2.
$$
Therefore \eqref{eq:20210422-1} holds from \eqref{eq:20210422-2}.
\end{proof}

\section{Conforming Virtual Element Method for a polyharmonic equation}\label{sec:vempolyharmonic}

In this section we will adopt the constructed conforming virtual elements to discretize the following polyharmonic equation with a lower order term: find $u\in H_0^m(\Omega)$ such that
\begin{equation}\label{eq:polyharmonic}
(\nabla^mu, \nabla^mv)+c(u, v)=(f, v)\quad\forall~v\in H_0^m(\Omega),
\end{equation}
where $f\in L^2(\Omega)$ and constant $c\geq0$.

\subsection{Conforming virtual element method}
Let the global virtual element space
$$
V_h:=\{v_h\in H_0^m(\Omega): v_h|_K\in V_{k}^{m}(K) \textrm{ for each } K\in\mathcal T_h\}.
$$

To define the discrete bilinear form, we first introduce the stabilization term
$$
S_K(w,v):=\sum_{r=1}^{n}\sum_{F\in\mathcal F^{r}(K)}\sum_{\alpha\in A_{r}\atop |\alpha|\leq m-1}h_K^{r+2|\alpha|-2m}\bigg(Q_{k-2m+|\alpha|}^{F}\frac{\partial^{|\alpha|}w}{\partial \nu_{F}^{\alpha}},Q_{k-2m+|\alpha|}^{F}\frac{\partial^{|\alpha|}v}{\partial \nu_{F}^{\alpha}}\bigg)_{F}
$$
for any $w,v\in V_{k}^{m}(K)$. When $k=m$, the stabilization term will reduce to
$$
S_K(w,v)=\sum_{\delta\in\mathcal F^n(K)}\sum_{j=0}^{m-1}h_K^{n+2j-2m}(\nabla^jw)(\delta):(\nabla^jv)(\delta).
$$
By 
\eqref{eq:normequivalence-kerpi1}, we acquire the norm equivalence of the stabilization term
\begin{equation}\label{eq:normequivalence-kerpikstab}
S_K(v-\Pi_k^Kv,v-\Pi_k^Kv)\eqsim |v-\Pi_k^Kv|_{m,K}^2\quad\forall~v\in V_{k}^{m}(K).
\end{equation}

Next define the linear form $a_h(\cdot, \cdot): V_h\times V_h\to\mathbb R$ as
$$
a_h(w_h, v_h):=\sum_{K\in\mathcal T_h}a_{h,K}(w_h, v_h),
$$
where
\begin{align*}
a_{h,K}(w_h, v_h):=&\,(\nabla^m\Pi_k^Kw_h, \nabla^m\Pi_k^Kv_h)_K+S_K(w_h-\Pi_k^Kw_h,v_h-\Pi_k^Kv_h) \\
&+c(Q_k^Kw_h, Q_k^Kv_h)_K.
\end{align*}
Clearly we obtain from \eqref{eq:localprojprop1} and \eqref{eq:projlocal1} that
\begin{equation}\label{eq:ahKprop1}
a_{h,K}(v, q)=(\nabla^mv, \nabla^m q)_K+c(v,  q)_K\quad\forall~v\in V_{k}^{m}(K), q\in \mathbb P_{k}(K).   
\end{equation}

\begin{lemma}
We have
\begin{equation}\label{eq:ahkcoercivity}
a_{h,K}(w, v)\lesssim (|w|_{m,K}+\|w\|_{0,K})(|v|_{m,K}+\|v\|_{0,K})\quad\forall~w,v\in V_{k}^{m}(K),
\end{equation}
\begin{equation}\label{eq:ahcoercivity}
a_{h}(v_h, v_h)\eqsim |v_h|_{m}^2\quad\forall~v_h\in V_h.   
\end{equation}
\end{lemma}
\begin{proof}
Let $v\in V_{k}^{m}(K)$.
By \eqref{eq:normequivalence-kerpikstab}, we get
\[
a_{h,K}(v, v)\eqsim |v|_{m,K}^2+c\|Q_k^Kv\|_{0,K}^2,
\]
which implies \eqref{eq:ahkcoercivity}, and \eqref{eq:ahcoercivity} by the Poincar\'e inequality.
\end{proof}


With previous preparations, we propose the following conforming virtual element method for the polyharmonic equation \eqref{eq:polyharmonic}: find $u_h\in V_h$ such that
\begin{equation}\label{Hmcfmvem}
a_{h}(u_h, v_h)=\langle f, v_h\rangle\quad\forall~v_h\in V_h,
\end{equation}
where $\langle f, v_h\rangle:=\sum\limits_{K\in\mathcal T_h}(f, Q_k^Kv_h)_K$.
Due to \eqref{eq:ahcoercivity}, the virtual element method \eqref{Hmcfmvem} is well-posed.

\subsection{Interpolation error estimate}
To derive the error estimate of the virtual element method \eqref{Hmcfmvem},
we define an interpolation operator $I_h: H_0^m(\Omega)\to V_h$ based on the DoFs \eqref{eq:cfmvemdof1}-\eqref{eq:cfmvemdof3}: for any $v\in H_0^m(\Omega)$, $I_hv\in V_h$ is determined by
$$
Q_{k-2m}^K(I_hv)=Q_{k-2m}^Kv,
$$
\begin{align*}
Q_{k-2m+|\alpha|}^F\frac{\partial^{|\alpha|}(I_hv)}{\partial \nu_F^{\alpha}}=\frac{1}{\# \mathcal T_F}\sum_{K\in\mathcal T_F}Q_{k-2m+|\alpha|}^F\frac{\partial^{|\alpha|}(Q_k^Kv)}{\partial \nu_F^{\alpha}}
\end{align*}
for each $K\in\mathcal T_h$, interior $F\in\mathcal F^{r}_h$, $r=1,\cdots,n$, $\alpha\in A_r$, and $|\alpha|\leq m-1$,
where $\mathcal T_F$ is the set of all $n$-dimensional polytopes in $\mathcal T_h$ sharing face $F$.

\begin{lemma}
Let $s\geq m$. It holds
\begin{equation}\label{eq:interpolationestimate}
\sum_{j=0}^mh^j|v-I_hv|_{j}\lesssim h^{\min\{s,k+1\}}|v|_{\min\{s,k+1\}}\quad\forall~v\in H^s(\Omega)\cap H_0^m(\Omega).   
\end{equation}
\end{lemma}
\begin{proof}
Since $Q_{k-2m}^K(Q_k^Kv-I_hv)=Q_{k-2m}^Kv-Q_{k-2m}^K(I_hv)=0$,
it follows from \eqref{eq:normequivalence1} and the definition of $I_hv$ that
\begin{align*}   
&\quad\sum_{K\in\mathcal T_h}|Q_k^Kv-I_hv|_{j,K}^2\\
&\lesssim \sum_{K\in\mathcal T_h}\sum_{r=1}^{n}\sum_{F\in\mathcal F^{r}(K)}\sum_{\alpha\in A_{r}, |\alpha|\leq m-1}h_K^{r+2|\alpha|-2j}\bigg\|Q_{k-2m+|\alpha|}^{F}\frac{\partial^{|\alpha|}(Q_k^Kv-I_hv)}{\partial \nu_{F}^{\alpha}}\bigg\|_{0,F}^2 \\
&\lesssim \sum_{K\in\mathcal T_h}\sum_{r=1}^{n}\sum_{F\in\mathcal F^{r}(K)}\sum_{\alpha\in A_{r}, |\alpha|\leq m-1}\sum_{K'\in\mathcal T_F}h_K^{r+2|\alpha|-2j}\bigg\|\frac{\partial^{|\alpha|}(Q_k^Kv)}{\partial \nu_{F}^{\alpha}}-\frac{\partial^{|\alpha|}(Q_k^{K'}v)}{\partial \nu_{F}^{\alpha}}\bigg\|_{0,F}^2 \\
&\lesssim \sum_{K\in\mathcal T_h}\sum_{r=1}^{n}\sum_{F\in\mathcal F^{r}(K)}\sum_{i=0}^{m-1}\sum_{K'\in\mathcal T_F}h_K^{r+2i-2j}\|\nabla^i(Q_k^Kv)-\nabla^i(Q_k^{K'}v)\|_{0,F}^2.
\end{align*}
By
the inverse inequality for polynomials~\eqref{eq:polyinverse}, the similar argument as in \cite[Lemma 3.3]{Wang2001} and \cite[Lemma 2.1]{BrennerWangZhao2004}, and the trace inequality \eqref{L2trace},
\begin{align*}
\sum_{K\in\mathcal T_h}|Q_k^Kv-I_hv|_{j,K}^2&\lesssim  \sum_{K\in\mathcal T_h}\sum_{F\in\mathcal F^{1}(K)}\sum_{i=0}^{m-1}\sum_{K'\in\mathcal T_F}h_K^{1+2i-2j}\|\nabla^i(Q_k^Kv)-\nabla^i(Q_k^{K'}v)\|_{0,F}^2 \\
&\lesssim  \sum_{K\in\mathcal T_h}\sum_{i=0}^{m-1}h_K^{1+2i-2j}\|\nabla^iv-\nabla^i(Q_k^Kv)\|_{0,\partial K}^2 \\
&\lesssim  \sum_{K\in\mathcal T_h}\sum_{i=0}^{m-1}h_K^{2i-2j}(|v-Q_k^Kv|_{i,K}^2+h_K^2|v-Q_k^Kv|_{i+1,K}^2).
\end{align*}
Hence
$$
\sum_{K\in\mathcal T_h}|Q_k^Kv-I_hv|_{j,K}^2\lesssim \sum_{K\in\mathcal T_h}\sum_{i=0}^{m}h_K^{2i-2j}|v-Q_k^Kv|_{i,K}^2.
$$
Then we achieve from the triangle inequality that
$$
|v-I_hv|_{j}^2\lesssim \sum_{K\in\mathcal T_h}\sum_{i=0}^{m}h_K^{2i-2j}|v-Q_k^Kv|_{i,K}^2.
$$
Thus \eqref{eq:interpolationestimate} holds from \eqref{eq:l2projectionestimate}.
\end{proof}

\subsection{Error estimate}

With the interpolation error estimate \eqref{eq:interpolationestimate}, we can present the a priori error estimate of the conforming virtual element method \eqref{Hmcfmvem}.
Define a global operator $\Pi_h: V_h\to L^2(\Omega)$ by $(\Pi_hv_h)|_K:=\Pi_k^K(v_h|_K)$ for each $K\in\mathcal T_h$.
For an element-wise smooth function $v$, let the usual squared broken semi-norm
$$
|v|_{m,h}^2:=\sum_{K\in\mathcal T_h}|v|_{m,K}^2.
$$

\begin{theorem}\label{errorestimate}
Let $u\in H^s(\Omega)\cap H_0^m(\Omega)$ with $s\geq m$ be the solution of the polyharmonic equation \eqref{eq:polyharmonic}, and $u_h\in V_h$ be the solution of the conforming virtual element method \eqref{Hmcfmvem}. Assume the mesh $\mathcal T_h$ satisfies conditions (A1) and (A2). Assume
$f\in H^m(\mathcal T_h)$. Then we have
\begin{equation}\label{eq:uherrorestimate}
|u-u_h|_m\lesssim h^{\min\{s,k+1\}-m}|u|_{\min\{s,k+1\}}+\textrm{osc}_h(f),
\end{equation}
\begin{equation}\label{eq:pihuherrorestimate}
|u-\Pi_hu_h|_{m,h}\lesssim h^{\min\{s,k+1\}-m}|u|_{\min\{s,k+1\}}+\textrm{osc}_h(f),
\end{equation}
where $\textrm{osc}_h^2(f):=\sum\limits_{K\in\mathcal T_h}h_K^{2m}\|f-Q_k^Kf\|_{0,K}^2$.
\end{theorem}
\begin{proof}
Let $v_h=I_hu-u_h\in H_0^m(\Omega)$ for simplicity.
Thanks to \eqref{eq:ahKprop1} and \eqref{eq:ahkcoercivity},
\begin{align*}
&\quad \; a_{h,K}(I_hu, v_h) - (\nabla^mu, \nabla^mv_h)_K-c(u, v_h)_K \\
& =a_{h,K}(I_hu-Q_k^Ku, v_h) - (\nabla^m(u-Q_k^Ku), \nabla^mv_h)_K-c(u-Q_k^Ku, v_h)_K \\
&\lesssim (|I_hu-Q_k^Ku|_{m,K}+\|I_hu-Q_k^Ku\|_{0,K})(|v_h|_{m,K}+\|v_h\|_{0,K}) \\
&\quad +(|u-Q_k^Ku|_{m,K}+\|u-Q_k^Ku\|_{0,K})(|v_h|_{m,K}+\|v_h\|_{0,K}) \\
&\lesssim (|u-I_hu|_{m,K}+\|u-I_hu\|_{0,K})(|v_h|_{m,K}+\|v_h\|_{0,K}) \\
&\quad +(|u-Q_k^Ku|_{m,K}+\|u-Q_k^Ku\|_{0,K})(|v_h|_{m,K}+\|v_h\|_{0,K}).
\end{align*} 
By \eqref{eq:l2projectionestimate}, it holds
\begin{align*}
(f, v_h)-\langle f, v_h\rangle&=\sum_{K\in\mathcal T_h}(f, v_h-Q_k^Kv_h)_K=\sum_{K\in\mathcal T_h}(f-Q_k^Kf, v_h-Q_k^Kv_h)_K \\
&\lesssim \sum_{K\in\mathcal T_h}h_K^{m}\|f-Q_k^Kf\|_{0,K}|v_h|_{m,K}\lesssim \textrm{osc}_h(f)|v_h|_{m}.
\end{align*} 
Combining the last two inequalities, we get from \eqref{eq:polyharmonic}, \eqref{eq:interpolationestimate}, \eqref{eq:l2projectionestimate} and the Poincar\'e inequality that
\begin{align*}
a_h(I_hu, v_h)-\langle f, v_h\rangle 
&= a_h(I_hu, v_h) - (\nabla^mu, \nabla^mv_h)-c(u, v_h)+(f, v_h)-\langle f, v_h\rangle \\
&\lesssim (h^{\min\{s,k+1\}-m}|u|_{\min\{s,k+1\}}+\textrm{osc}_h(f))|v_h|_{m}.
\end{align*} 
Then we acquire from
\eqref{eq:ahcoercivity} and \eqref{Hmcfmvem} that
\begin{align*}
|v_h|_m^2 & \eqsim a_h(I_hu-u_h, v_h)=a_h(I_hu, v_h)-\langle f, v_h\rangle \\
&\lesssim (h^{\min\{s,k+1\}-m}|u|_{\min\{s,k+1\}}+\textrm{osc}_h(f))|v_h|_{m}.
\end{align*}
As a result,
$$
|I_hu-u_h|_m \lesssim h^{\min\{s,k+1\}-m}|u|_{\min\{s,k+1\}}+\textrm{osc}_h(f).
$$
Therefore \eqref{eq:uherrorestimate} holds from the last inequality and \eqref{eq:interpolationestimate}.

Next we prove \eqref{eq:pihuherrorestimate}. By \eqref{eq:localprojprop1}, on each $K\in\mathcal T_h$ we have
$$
u-\Pi_k^Ku_h=u-u_h+\Pi_k^K(Q_k^Ku-u_h)-(Q_k^Ku-u_h).
$$
Then it follows from \eqref{eq:normequivalence-kerpi0} and the triangle inequality that
\begin{align*}
|u-\Pi_hu_h|_{m,h}^2&\lesssim |u-u_h|_{m}^2 +\sum_{K\in\mathcal T_h}|\Pi_k^K(Q_k^Ku-u_h)-(Q_k^Ku-u_h)|_{m,K}^2 \\
&\lesssim |u-u_h|_{m}^2+\sum_{K\in\mathcal T_h}|u-Q_k^Ku|_{m,K}^2.
\end{align*}
Finally we arrive at \eqref{eq:pihuherrorestimate} from \eqref{eq:uherrorestimate} and \eqref{eq:l2projectionestimate}.
\end{proof}

Under the assumption that the partition $\mathcal T_h$ is quasi-uniform and $h$ is sufficiently small, we can show that the condition number of the resulting coefficient matrix of the conforming virtual element method \eqref{Hmcfmvem} is $O(h^{2m})$, whose order is only related to the order of the differential operator. See also Section 3.4 in \cite{BeiraodaVeigaDassiRusso2020}.

\section{Numerical results}\label{sec:numericalexamps}

In this section, we provide two examples to numerically verify
the convergence of the $H^m$-conforming virtual element method \eqref{Hmcfmvem} with $c=1$.
Let $\Omega = (0, 1)\times(0, 1)$. And the rectangular domain $\Omega$ is partitioned by the convex polygonal mesh $\mathcal T_0$ and non-convex polygonal mesh $\mathcal T_1$ respectively, shown in Figure~\ref{fig:mesh}.  
The numerical examples are implemented by using the FEALPy package~\cite{fealpy}.

\begin{figure}[htbp]
\centering
\begin{minipage}[t]{0.49\linewidth}
\centering
\includegraphics[width=5cm]{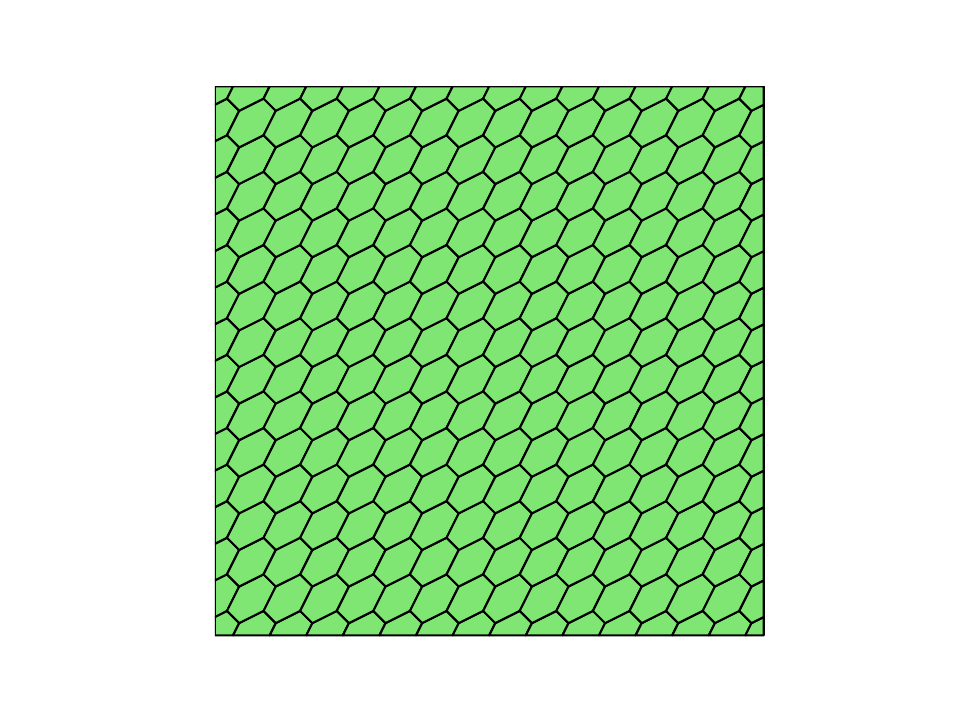}
\end{minipage}%
\begin{minipage}[t]{0.49\linewidth}
\centering
\includegraphics[width=5cm]{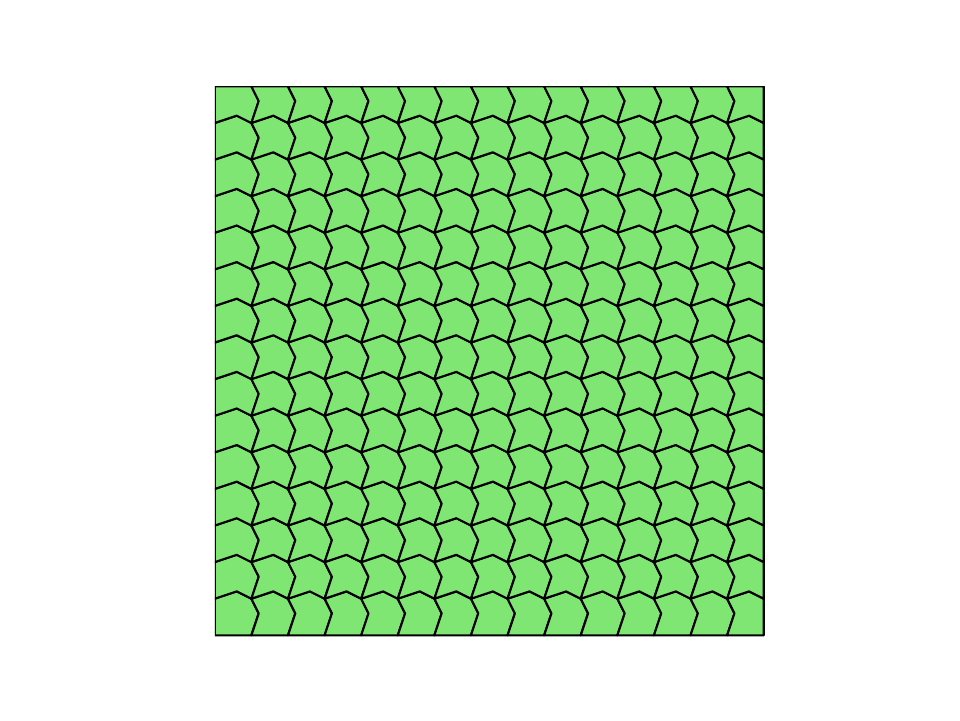}
\end{minipage}%
\centering
\caption{Convex polygon mesh $\mathcal T_0$(left) and non-convex polygon mesh 
$\mathcal T_1$(right).}
\label{fig:mesh}
\end{figure}

\begin{example}\label{exam1}
Consider polyharmonic equation \eqref{eq:polyharmonic} with $m = 2$. Take the exact solution $u = \sin^2(\pi x)\sin^2(\pi y)$, and the right-hand side $f$ is computed from polyharmonic equation \eqref{eq:polyharmonic}.
\end{example}

Choose $k = 2, 3, 4, 5$ for the virtual element method \eqref{Hmcfmvem}. 
The numerical results are listed in Figure \ref{fig:H2error}.
We can see that $|u - \Pi_h u_h|_{2, h} = O(h^{k-1})$, which
coincides with Theorem \ref{errorestimate}. 

\begin{figure}[htbp]
\centering
\begin{minipage}[t]{0.49\linewidth}
\centering
\includegraphics[width=5cm]{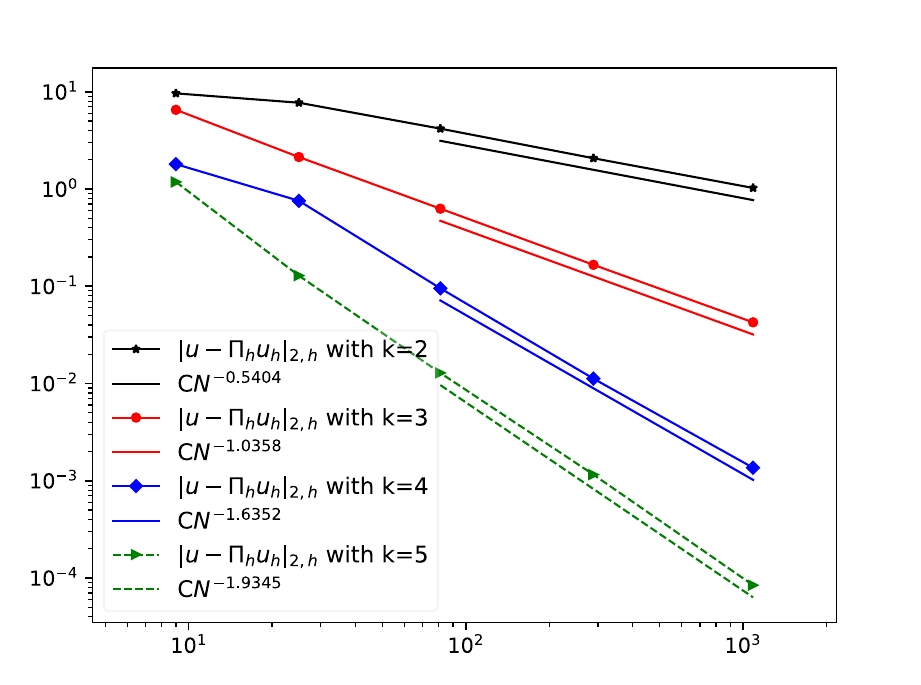}
\end{minipage}%
\begin{minipage}[t]{0.49\linewidth}
\centering
\includegraphics[width=5cm]{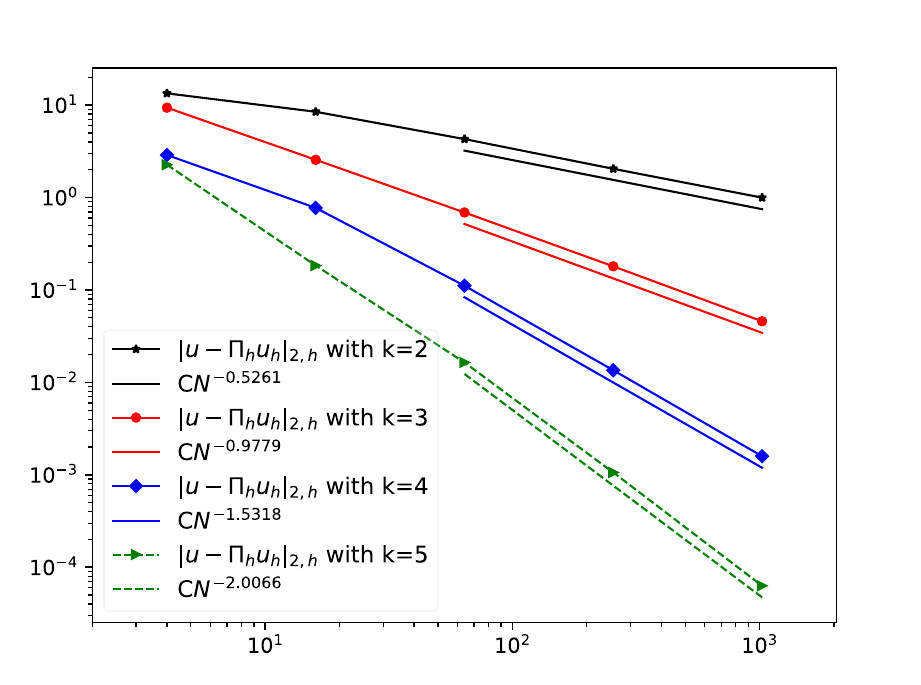}
\end{minipage}%
\centering
\caption{Error $|u - \Pi_h u_h|_{2, h}$ of Example
    \ref{exam1} with $m=2$ on convex polygon mesh $\mathcal T_0$(left) 
and non-convex polygon mesh $\mathcal T_1$(right) with $k = 2, 3, 4, 5$.}
\label{fig:H2error}
\end{figure} 

\begin{example}\label{exam2}
Consider polyharmonic equation \eqref{eq:polyharmonic} with $m = 3$. Take the exact solution $u = \sin^3(\pi x)\sin^3(\pi y)$, and the right-hand side $f$ is computed from polyharmonic equation \eqref{eq:polyharmonic}.
\end{example}

In this example we set $k = 3, 4, 5, 6$, 
and present numerical results in Figure~\ref{fig:H3error}.
We observe from Figure \ref{fig:H3error} that $|u - \Pi_h u_h|_{3, h} = O(h^{k-2})$, which again agrees with Theorem \ref{errorestimate}. 

\begin{figure}[htbp]
\centering
\begin{minipage}[t]{0.49\linewidth}
\centering
\includegraphics[width=5cm]{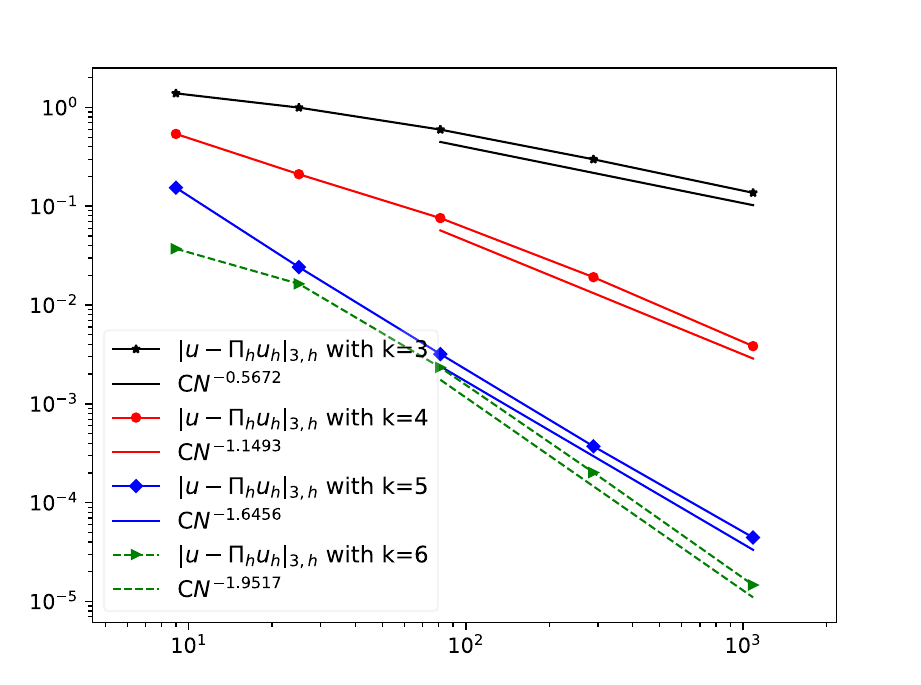}
\end{minipage}%
\begin{minipage}[t]{0.49\linewidth}
\centering
\includegraphics[width=5cm]{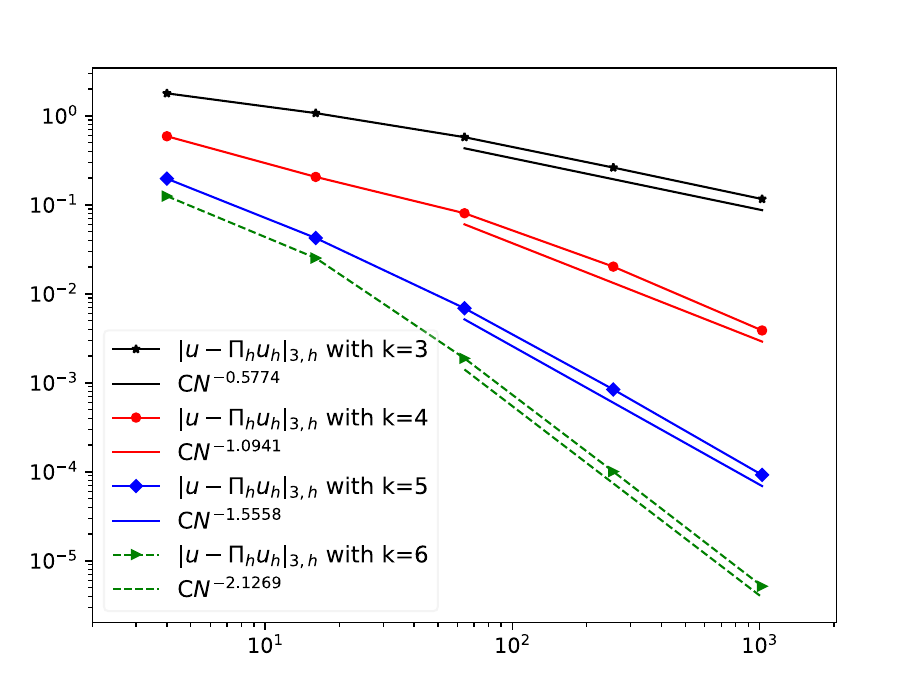}
\end{minipage}%
\centering
\caption{Error $|u - \Pi_h u_h|_{3, h}$ of Example
    \ref{exam2} with $m=3$ on convex polygon mesh $\mathcal T_0$(left) 
and non-convex polygon mesh $\mathcal T_1$(right) with $k = 3, 4, 5, 6$.}
\label{fig:H3error}
\end{figure}

\section*{Acknowledgements}

The author would like to thank Prof. Long Chen in University of California, Irvine for the insightful discussion.

\bibliographystyle{abbrv}
\bibliography{./HmcfmVEM}
\end{document}